\documentclass[a4paper,reqno,final]{amsart}
\usepackage[foot]{amsaddr}
\usepackage{mathtools,amssymb}
\usepackage{fullpage,enumitem}
\usepackage{showkeys}
\RequirePackage[l2tabu,orthodox]{nag}
\usepackage[all, warning]{onlyamsmath}
\usepackage[bookmarks=false,draft=false,breaklinks,colorlinks]{hyperref}
\numberwithin{equation}{section}
\allowdisplaybreaks

\newenvironment{Ack}%
{\par \vspace{\baselineskip}%
 \noindent \textbf{Acknowledgements.}}%
{\par \vspace{\baselineskip}}

\setenumerate{nosep}

\usepackage{cleveref}
\crefname{thm}{Theorem}{Theorems}
\crefname{dfn}{Definition}{Definitions}
\crefname{prp}{Proposition}{Propositions}
\crefname{lem}{Lemma}{Lemmas}
\crefname{cor}{Corollary}{Corollaries}
\crefname{rmk}{Remark}{Remarks}
\crefname{figure}{Figure}{Figures}
\crefname{table}{Table}{Tables}
\crefname{section}{\S\!}{\S\S\!}
\crefname{subsection}{\S\!}{\S\S\!}
\crefname{subsubsection}{\S\!}{\S\S\!}
\crefname{equation}{}{}

\usepackage{amsthm}
\theoremstyle{definition}
\newtheorem{thm}{Theorem}[section]
\newtheorem{dfn}[thm]{Definition}
\newtheorem{prp}[thm]{Proposition}
\newtheorem{lem}[thm]{Lemma}
\newtheorem{cor}[thm]{Corollary}
\newtheorem{rmk}[thm]{Remark}

\newcommand{\ol}{\overline}
\newcommand{\wh}{\widehat}
\newcommand{\wt}{\widetilde}

\newcommand{\bs}{\backslash}
\newcommand{\ceq}{\coloneqq}

\newcommand{\lto}{\longrightarrow}

\newcommand{\inj}{\hookrightarrow}
\newcommand{\srj}{\twoheadrightarrow}

\newcommand{\lsto}{\xrightarrow{\ \sim \ }}

\newcommand{\lact}{\curvearrowleft}
\newcommand{\ract}{\curvearrowright}
\newcommand{\midd}{\, | \,}
\newcommand{\Fq}{\bbF_q}
\newcommand{\bbC}{\mathbb{C}}
\newcommand{\bbF}{\mathbb{F}}
\newcommand{\bbN}{\mathbb{N}}
\newcommand{\bbP}{\mathbb{P}}
\newcommand{\bbQ}{\mathbb{Q}}
\newcommand{\bbR}{\mathbb{R}}

\newcommand{\bbZ}{\mathbb{Z}}
\newcommand{\frS}{\mathfrak{S}}

\newcommand{\fsl}{\mathfrak{sl}}

\newcommand{\clC}{\mathcal{C}}
\newcommand{\clE}{\mathcal{E}}
\newcommand{\clF}{\mathcal{F}}
\newcommand{\clL}{\mathcal{L}}
\newcommand{\clN}{\mathfrak{N}}
\newcommand{\clU}{\mathcal{U}}
\newcommand{\clV}{\mathcal{V}}
\newcommand{\clW}{\mathcal{W}}
\newcommand{\sfP}{\mathsf{P}}

\DeclareMathOperator{\jz}{j_z}
\DeclareMathOperator{\js}{\mathbf{j}^2}

\DeclareMathOperator{\id}{id}

\DeclareMathOperator{\Gr}{Gr}
\DeclareMathOperator{\GL}{GL}

\DeclareMathOperator{\SU}{SU}

\DeclareMathOperator{\Ind}{Ind}

\newcommand{\fl}[1]{\lfloor #1 \rfloor}

\newcommand{\hg}[2]{{}_{#1} F_{#2}}
\newcommand{\abs}[1]{\left| #1 \right|}

\newcommand{\bra}[1]{\left< #1 \right|}
\newcommand{\ket}[1]{\left| #1 \right>}

\newcommand{\rst}[2]{\left. #1 \right|_{#2}}
\newcommand{\qhg}[2]{{}_{#1} \phi_{#2}}
\newcommand{\pair}[2]{\left< #1 \, \middle| \, #2 \right>}

\newcommand{\qbinom}[3]{\genfrac{[}{]}{0pt}{}{#1}{#2}_{#3}}

\newmuskip\HGmuskip
\newcommand{\HGcomma}{{\normalcomma}\mskip\HGmuskip}
\newcommand*\HG[6][8]{%
 \begingroup
  \HGmuskip=#1mu\relax
  \mathchardef\normalcomma=\mathcode`,
  \mathcode`\,=\string"8000
  \begingroup\lccode`\~=`\,
  \lowercase{\endgroup\let~}\HGcomma
  {}_{#2}F_{#3}{\left[\genfrac..{0pt}{}{#4}{#5};#6\right]}%
 \endgroup}
\newcommand*\qHG[7][8]{%
 \begingroup
  \HGmuskip=#1mu\relax
  \mathchardef\normalcomma=\mathcode`,
  \mathcode`\,=\string"8000
  \begingroup\lccode`\~=`\,
  \lowercase{\endgroup\let~}\HGcomma
  {}_{#2}\phi_{#3}{\left[\genfrac..{0pt}{}{#4}{#5};#6,#7\right]}%
 \endgroup}

\newenvironment{sbm}{\left[\begin{smallmatrix}}{\end{smallmatrix}\right]}
\newenvironment{nouppercase}{%
  \renewcommand{\uppercasenonmath}[1]{}}{}

\begin{document}

\title{$q$-Racah probability distribution}
\author{Masahito Hayashi$^{1,2,3}$}
\address{$^1$ School of Data Science, The Chinese University of Hong Kong,
 Shenzhen, Longgang District, Shenzhen, 518172, China.}
\address{$^2$ International Quantum Academy (SIQA), Shenzhen 518048, China.}
\address{$^3$ Graduate School of Mathematics, Nagoya University, 
 Furo-cho, Chikusa-ku, Nagoya, 464-8602, Japan.}
\email{hmasahito@cuhk.edu.cn} %masahito@math.nagoya-u.ac.jp}
\author{AKIHITO HORA$^4$} %{Akihito Hora}
\address{$^4$ Department of Mathematics, Faculty of Science, Hokkaido University, %
 Kita 10, Nishi 8, Kita-Ku, Sapporo, Hokkaido, 060-0810, Japan.}
\email{hora@math.sci.hokudai.ac.jp} 
\author{SHINTAROU YANAGIDA$^3$} %{Shintarou Yanagida}
\email{yanagida@math.nagoya-u.ac.jp}
%
%\date{2021.04.25, 2023.10.13-11.11, 2024.01.20-02.14, 03.26}
\date{2024.06.04}

\begin{abstract}
We introduce a certain discrete probability distribution $P_{n,m,k,l;q}$ 
having non-negative integer parameters $n,m,k,l$ and quantum parameter $q$
which arises from a zonal spherical function of 
the Grassmannian over the finite field $\Fq$ with a distinguished spherical vector. 
Using representation theoretic arguments and hypergeometric summation technique,
we derive the presentation of the probability mass function by a single $q$-Racah polynomial, 
and also the presentation of the cumulative distribution function 
in terms of a terminating ${}_4 \phi_3$-hypergeometric series.
\end{abstract}

\begin{nouppercase}
\maketitle
\end{nouppercase}

{\small \tableofcontents}

%%%%%%%%%%%%%%%%%%%%%%%%%%%%%%%%%%%%%%%%%%%%%%%%%%%%%%%%%%%%%%%%%%%%%%%%%%%%%%%%%%%%%%%%%%
%%%%%%%%%%%%%%%%%%%%%%%%%%%%%%%%%%%%%%%%%%%%%%%%%%%%%%%%%%%%%%%%%%%%%%%%%%%%%%%%%%%%%%%%%%
\section{Introduction}\label{s:intro}

This article is an excerpt from the manuscript \cite{HHY} posted on the arXiv,
and presents the construction of a certain multi-parameter discrete probability distribution
expressed in terms of ($q$-)Racah polynomial.
The other parts of \cite{HHY} give asymptotic analysis of the distribution and 
applications to quantum information theory, which will appear in the sequel \cite{HHY2,H3}.

%%%%%%%%%%%%%%%%%%%%%%%%%%%%%%%%%%%%%%%%%%%%%%%%%%%%%%%%%%%%%%%%%%%%%%%%%%%%%%%%%%%%%%%%%%
\subsection{Racah and $q$-Racah polynomials}\label{ss:intro:RqR}

The family of Racah polynomials is an orthogonal polynomial system 
introduced by Wilson \cite{W} as a generalization of Racah coefficients \cite{Ra} 
and Wigner's 6j symbols \cite{Wig} in quantum mechanics. 
It sits in the top line of the Askey scheme of hypergeometric orthogonal polynomials
(\cite[p.XXXVII]{LNM}, \cite[p.183]{KLS}).
The fact that the Racah polynomial is a hypergeometric orthogonal polynomial 
enables us to understand in a uniform way the various formulas of 6j symbols such as 
symmetry of parameters, orthogonality and recurrence relations. %and generating functions.

In \cite{AW1}, Askey and Wilson introduced a generalization of Racah polynomials,
now called $q$-Racah polynomials. 
They are $\qhg{4}{3}$ basic hypergeometric polynomials, and sit in the top line of 
the $q$-Askey scheme of orthogonal polynomials (\cite[p.413]{KLS}).
Thus, we can regard $q$-Racah polynomials as one of the master classes of orthogonal polynomials.
See \cite{KLS} for the encyclopedic explanation of Racah and $q$-Racah polynomials.
Let us also refer to \cite{Ro} for a nice review on these classical subjects 
and for the recent progress in elliptic hypergeometric functions.

Since their appearance, Racah and $q$-Racah polynomials have attracted a great deal of interest 
in various fields of mathematics.
One of the sources of interest is the wide range of relations to representation theory of 
Lie groups and quantum groups. 
An example of such relations is the one to 6j symbols mentioned in the beginning,
related to the representation theory of the Lie group $\SU(2)$.
This relation has a natural $q$-analogue:
The quantum $6j$-symbols of the quantum enveloping algebra $U_q(\fsl(2))$ 
are expressed in terms of $q$-Racah polynomials \cite{KiR}.
$q$-Racah polynomials also appear in connection with the $\SU(2)$ dynamical quantum group \cite{KR}.
Also, in the context of integrable probability theory, $q$-Racah polynomials appear in 
the Boltzmann weights and in the relevant probability distributions \cite{CP,Man}.

The appearance of $q$-Racah polynomials in these representation-theoretic contexts
has its origin in the corresponding $R$-matrices of quantum groups \cite{Dr,F} 
or quantum integrability \cite{YB}.
Such quantum integrability together with properties of 
basic hypergeometric orthogonal polynomials \cite{AAR,AW2,GR}
suggest that one can do exact analysis on the system where $q$-Racah polynomials appear.
%Here the quantum parameter $q$ appears as the deformation parameter 
%of the concerned algebra structure, 
%typically quantum enveloping algebras (of finite and type) \cite{CP,L}, 
%and the dynamical (or elliptic) cousins \cite{F,FV,EV}.

Apart from quantum groups, 
there are several other sources of $q$-deformation in representation theory.
The one relevant to this article is 
the replacement of sets by vector spaces on a finite field $\bbF_q$.
A simple example is the Gaussian coefficient $\qbinom{n}{m}{q}$ which counts 
$\bbF_q$-rational points in the Grassmannian $\Gr(m,n)$ over $\bbF_q$.
%$m$-dimensional subspaces of the vector space $\bbF_q^n$.
This ad-hoc procedure sometimes gives non-trivial consequences:
Consider the Gelfand pair $(G,K)=(\GL(n,\bbF_q),P(m,n-m,\bbF_q))$ of Chevalley-groups
whose homogeneous space $G/K$ is $\Gr(m,n)$.
A classical result due to Delsarte \cite{De1,De2} says that 
the zonal spherical functions for $\Gr(m,n)$ is expressed in terms of a $q$-Hahn polynomial,
which is a $\qhg{3}{2}$ basic hypergeometric polynomial and obtained by 
a degeneration limit of a $q$-Racah polynomial (see \cref{ss:q:sph} for details).

%%%%%%%%%%%%%%%%%%%%%%%%%%%%%%%%%%%%%%%%%%%%%%%%%%%%%%%%%%%%%%%%%%%%%%%%%%%%%%%%%%%%%%%%%%%%%%%%%%%%
\subsection{Features of this article}\label{ss:intro:f}

This article reports on a new appearance of $q$-Racah polynomials in the intersection of 
probability theory, the theory of special functions, and representation theory.
We consider a function $p(x;q) = p(x \midd n,m,k,l;q)$ 
expressed by $q$-Racah polynomial (\cref{dfn:0:p-genq}),
depending on non-negative integer parameters $n,m,k,l$, a deformation parameter $q$,
and the variable $x \in \{0,1,\dotsc,m\}$.
The main statement (\cref{thm:main}) is that 
$p(x;q)$ is the probability mass function (pmf for short) of a 
discrete probability distribution $P_{n,m,k,l;q}$ on the set $\{0,1,\dotsc,m\}$ 
in the cases where  (1) $q$ is a prime power, (2) the $q \to 1$ limit, and 
(3) $q \in \bbR$ and $0 < \abs{q-1}\ll 1$.
We call $P_{n,m,k,l;q}$ the $q$-Racah probability distribution.

The proof is, roughly speaking, based on hypergeometric summation techniques 
and tools from representation theory.
The former is a classical topic in the theory of special functions, 
and there are known many results named ``hypergeometric summation formulas".
In \cref{s:sum}, we give two such formulas concerning terminating $\qhg{4}{3}$-series.
One of them (\cref{thm:q:sum=1}) yields the identity $\sum_{x=0}^m p(x;q)=1$ for any $q$.

Representation theory will be used in the proof of the positivity $p(x;q) \ge 0$.
In the case (1) $q$ is a prime power (\cref{ss:q:pos}),
 we consider the general linear group $\GL(n,\Fq)$ 
over a finite field $\bbF_q$ of order $q$, 
and introduce a unitary representation $T=(\bbC^2)^{\otimes [n]_q}$ of $\GL(n,\Fq)$,
where $[n]_q \ceq 1+q+\dotsb+q^{n-1}$ is the $q$-integer. %(see \eqref{eq:qntn:qint}).
Then we construct a positive quantity $\wt{p}(x;q) \ge 0$ using the irreducible decomposition.
Finally, we show $p(x;q)=\wt{p}(x;q)$ using the zonal spherical function associated to 
the Gelfand pair $(\GL(n,\Fq),P(m,n-m,\Fq))$, or the Grassmannian $\Gr(m,n)$ over $\Fq$,
and obtain $p(x;q) \ge 0$. %See \cref{ss:q:pos} for the detail.

In \cref{ss:q=1:pos},
the positivity in the case (2) $q \to 1$ is shown in a similar way as the case (1), 
but by considering a permutation action of the symmetric group $\frS_n$ 
on the tensor space $(\bbC^2)^{\otimes n}$, 
which appears in the classical Schur-Weyl duality between $\frS_n$ and $\SU(2)$.
The relevant Gelfand pair is $(\frS_n,\frS_m \times \frS_{n-m})$.
%See for the detail.
This Schur-Weyl duality will also be used in the proof of the positivity 
in the case (3) $0<\abs{q-1}\ll 1$. See \cref{s:q-small} for the detail.

%It is surprising that these representation theoretic gadgets 
%give rise to a non-trivial distribution which has some applications as mentioned above.
%The Schur-Weyl duality also plays an important role 
%in the proof of the main \cref{thm:main} \eqref{i:main:3}, 
%treating the case where $q$ is a real number such that $q \ne 1$ and $\abs{q-1} \ll 1$.

Our multi-parameter distribution $P_{n,m,k,l;q}$ has an interesting property: 
Both the pmf and the cumulative distribution function (cdf for short)
are given by a single terminating $\qhg{4}{3}$-series (see \cref{thm:0:cdf=qhg} for the cdf).
We cannot find such a distribution in the literature.
%In the language of hypergeometric functions, this property can be reformulated as 
%a new basic hypergeometric summation formula involving $q$-Racah polynomials (\cref{cor:q:sum}).
By this property, 
together with the fact that we cannot add extra parameters to $q$-Racah polynomials 
while keeping the orthogonality,
we may claim that our distribution $P_{n,m,k,l;q}$ is a master class 
of ``special discrete probability distributions'', 
which sounds similar to the fact that $q$-Racah polynomials form 
a master class of orthogonal polynomials.

The $q \to 1$ limit distribution $P_{n,m,k,l}$ studied in \cref{s:q=1} 
%which is only a special case from the point of view of the theory of special functions,
is very important from the point of view of application.
Using the well-known three-term recursion of Racah polynomial (\cref{prp:rec}),
we can do fine asymptotic analysis of $P_{n,m,k,l}$ in the large parameter limit.
Taking different scaling limits, we obtain normal, geometric, 
Rayleigh distributions and a convolution of binary distributions 
as the limit continuous distribution. 
The detail of the asymptotic analysis of $P_{n,m,k,l}$ 
will be given in the sequel \cite{HHY2}.
This analysis has important applications in quantum information theory,
presented in \cite{H3}.

Back to the point of view of representation theory and quantum integrable systems,
an interesting feature of our realization of $q$-Racah polynomial is 
that it is not directly related to quantum groups, but nevertheless we can do fine analysis.
As mentioned in \cref{ss:intro:RqR}, $q$-Racah polynomials appearing in the literature
\cite{CP,KiR,KR,Man} are more or less related to the $R$-matrices of quantum groups,
so that the origin of the appearance of $q$-Racah polynomials in this article 
is different from those papers.
The difference can also be checked by directly comparing our parameter range with theirs.
It implies that our setting has a hidden ``quantum integrability'' 
which is not encoded by the known quantum $R$-matrices or their cousins.

%%%%%%%%%%%%%%%%%%%%%%%%%%%%%%%%%%%%%%%%%%%%%%%%%%%%%%%%%%%%%%%%%%%%%%%%%%%%%%%%%%%%%%%%%%%%%%%%%%%%
\subsection{Summary of results}\label{ss:0:qR}

Let us explain our main results briefly.
%We consider non-negative integers $K,L,M,N \in \bbZ_{\ge 0}$ 
Let $(n,m,k,l)$ be a tuple of integers belonging to the set 
\begin{align}\label{eq:0:clN'}
 \clN' \ceq \{(n,m,k,l) \in \bbZ_{\ge 0}^4 \mid
  m \le n, \, k \le n, \ m+k-n \le l \le m \wedge k\},
\end{align}
where %$\fl{n/2}$ denotes the largest integer less or equal to $n/2$, and 
$m \wedge k \ceq \min\{m,k\}$.
We put $M \ceq m-l$ and $N \ceq n-m-k+l$. 
The set $\clN'$ parameterizes the four non-negative integers $k-l,l,M$ and $N$ whose sum is $n$.

\begin{dfn}\label{dfn:0:p-genq}
For $(n,m,k,l) \in \clN'$ and $x \in \{0,1,\dotsc,m \wedge (n-m)\}$, 
we define a rational function $p(x;q) = p(x \midd n,m,k,l;q) \in \bbQ(q)$ by 
\begin{align}\label{eq:0:p}
 p(x;q) \ceq
 \qbinom{n-k}{m-l}{q}\frac{\qbinom{n}{x}{q}}{\qbinom{n}{m}{q}} q^x \frac{[n-2x+1]_q}{[n-x+1]_q}
 \qHG{4}{3}{q^{-x},q^{x-n-1},q^{-M},q^{-N}}{q^{-m},q^{m-n},q^{-M-N}}{q}{q}.
\end{align}
Here we used the standard notation for $q$-binomial coefficients and 
$q$-hypergeometric series (see \cref{ss:0:ntn}).
\end{dfn}

Note that the $\qhg{4}{3}$-series in \eqref{eq:0:p} is terminating,
and hence $p(x;q)$ is indeed a rational function of $q$.
The function $p(x;q)$ has the symmetry 
\begin{align}
 p(x \midd n,m,k,l;q)=p(x \midd n,n-m,k,k-l;q).
\end{align}
Hence we will assume $m \le n-m$, i.e., consider the restricted parameter set
\begin{align}\label{eq:0:clN} 
 \clN \ceq \{(n,m,k,l) \in \bbZ_{\ge 0}^4 \mid
  m \le \fl{n/2}, \, k \le n, \ m+k-n \le l \le m \wedge k\},
\end{align}
where $\fl{n/2}$ denotes the largest integer less or equal to $n/2$,
and assume $x \in \{0,1,\dotsc,m\}$.

Using $q$-Racah polynomial \cite[\S14.2]{KLS}:
\begin{align}\label{eq:qRacah}
 R_s(\mu(z);a,b,c,d;q) \ceq
 \qHG{4}{3}{q^{-s},abq^{s+1},q^{-z},cdq^{z+1}}{aq,bdq,cq}{q}{q}, 
\end{align}
we can rewrite the function $p(x;q)$ as 
\begin{align}
\label{eq:0:p=qR}
  p(x;q) %p(x \midd n,m,k,l;q) 
 =\qbinom{n-k}{m-l}{q}\frac{\qbinom{n}{x}{q}}{\qbinom{n}{m}{q}} q^x \frac{[n-2x+1]_q}{[n-x+1]_q}
   R_x(\mu(M); q^{-m-1},q^{-n+m-1},q^{-n+k-1},1;q).
\end{align}
We call \eqref{eq:0:p=qR} the $q$-Racah presentation of $p(x;q)$.
We have another presentation of $p(x;q)$.

\begin{thm}[$q$-Hahn presentation, \cref{thm:sum1}]
For $(n,m,k,l) \in \clN$ and $x\in\{0,1,\dotsc,m\}$, we have the following identity in $\bbQ(q)$.
\begin{align}\label{eq:0:p=sum-qH}
 p(x;q) = q^x \frac{\qbinom{n}{x}{q}}{\qbinom{n}{m}{q}} \frac{[n-2x+1]_q}{[n-x+1]_q}
  \sum_{i=0}^x q^{i^2} \qbinom{M}{i}{q} \qbinom{N}{i}{q} 
  \qHG{3}{2}{q^{-i},q^{-x},q^{x-n-1}}{q^{-m},q^{m-n}}{q}{q}.
\end{align}
\end{thm}

The $\qhg{3}{2}$ in \eqref{eq:0:p=sum-qH} is a $q$-Hahn polynomial \cite[\S14.6]{KLS}.
See \eqref{eq:q:qHahn} and around for the detail.

We also have the well-defined $\lim_{q \to 1} p(x;q)$,
which will be denoted by $p(x) = p(x \midd n,m,k,l)$.
See \eqref{eq:q=1:p} for the explicit formula.
%\begin{align}\label{eq:0:q=1}
%\begin{split}
% p(x)
%&= p(x \midd n,m,k,l) \ceq \lim_{q \to 1}p(x \midd n,m,k,l;q) \\
%&=\binom{n-k}{m-l} \frac{\binom{n}{x}}{\binom{n}{m}} \frac{n-2x+1}{n-x+1} 
%  \HG{4}{3}{-x,x-n-1,-M,-N}{-m,m-n,-M-N}{1},
%\end{split}
%\end{align}
%where we used the standard symbols of binomial coefficients and 
%hypergeometric series (see \cref{ss:0:ntn}).
%We have $p(x) \in \bbQ$.
%Similarly as \eqref{eq:0:p=qR}, the function $p(x)$ has another presentation
%\eqref{eq:q=1:4332}:
%\begin{align}%\label{eq:0:q=1:H}
% p(x) = \frac{\binom{n}{x}}{\binom{n}{m}} \, \frac{n-2x+1}{n-x+1} \, 
% \sum_{i=0}^{M \wedge N} \binom{M}{i} \binom{N}{i} \, \HG{3}{2}{-i,-x,x-n-1}{-m,m-n}{1}.
%\end{align}
The functions $p(x;q)$ and $p(x)$ have interesting summation properties.
%The summation over $x$ is also expressed by a $q$-hypergeometric polynomial.

\begin{thm}[{\cref{thm:q:sum=1}}]\label{thm:0:cdf=qhg}
For $(n,m,k,l)\in\clN$ and $x \in\{0,1,\dotsc,m\}$, we have 
\begin{align}\label{eq:0:cdf=qhg}
 \sum_{u=0}^x p(u;q) %p(u \midd n,m,k,l;q) 
=\qbinom{n-k}{m-l}{q} \frac{\qbinom{n}{x}{q}}{\qbinom{n}{m}{q}} 
 \qHG{4}{3}{q^{-x},q^{x-n},q^{-M},q^{-N}}{q^{-m},q^{m-n},q^{-M-N}}{q}{q}
\end{align}
as an equality in $\bbQ(q)$.
Moreover, we have 
\begin{align}\label{eq:0:sum=1}
 \sum_{x=0}^m p(x;q)=1.
\end{align}
\end{thm}

Similar formulas hold for the limit function $p(x)$ (\cref{thm:q=1:sum=1}).
In particular, we have
\begin{align}\label{eq:0:q=1:sum=1}
 \sum_{x=0}^m p(x)=1.
\end{align}

The main result of this paper is:

\begin{thm}[$q$-Racah probability distribution]\label{thm:main}
Let $q$ be either
\begin{enumerate}
\item \label{i:main:1} a prime power, 
\item \label{i:main:2} $q=1$, or 
\item \label{i:main:3} a real number with $q \ne 1$ and $\abs{q-1}$ small enough.
\end{enumerate}
In the cases \eqref{i:main:1} and \eqref{i:main:3}, 
for $(n,m,k,l) \in \clN$, the function $p(x;q)=p(x \midd n,m,k,l;q)$ in \cref{dfn:0:p-genq}
is the pmf for a discrete probability distribution 
$P_{n,m,k,l;q}$ with variable $x\in\{0,1,\dotsc,m\}$.
The cdf $P_{n,m,k,l;q}[X \le x] \ceq \sum_{u=0}^x p(u \midd n,m,k,l;q)$ 
is expressed by the terminating $\qhg{4}{3}$-series \eqref{eq:0:cdf=qhg}.

In the case \eqref{i:main:2}, 
%similar statements hold for the limit function $p(x)=p(x \midd n,m,k,l)$.
%See \cref{thm:q=1:main} for the exact statement.
the function $p(x) = p(x \midd n,m,k,l)$ in \eqref{eq:q=1:p} is the pmf for a discrete
probability distribution $P_{n,m,k,l}$ with variable $x \in \{0,1,\dotsc,m\}$. 
The cdf $P_{n,m,k,l}[X \le x] \ceq \sum_{u=0}^x u=0 p(u \midd n,m,k,l)$ 
is expressed by the $\hg{4}{3}$-series \eqref{eq:q=1:cdf}.
\end{thm}

We call the distribution $P_{n,m,k,l};q$ and $P_{n,m,k,l}$ the $q$-Racah probability distribution
and the Racah probability distribution, respectively.

\begin{rmk}[added after the referees' comments]\label{rmk:ref1}
The case \eqref{i:main:1} of \cref{thm:main} was pointed out to us by the referee.
\end{rmk}

Considering the summation equalities \eqref{eq:0:sum=1} and \eqref{eq:0:q=1:sum=1}, 
the essential point of the proof is the positivity of the functions $p(x;q)$ and $p(x)$
for $x \in \{0,1,\dotsc,m\}$ in each case \eqref{i:main:1}--\eqref{i:main:3}.
The case \eqref{i:main:1} will be shown in \cref{s:q}, 
the case \eqref{i:main:2} in \cref{s:q=1},
and the case \eqref{i:main:3} in \cref{s:q-small}.

%%%%%%%%%%%%%%%%%%%%%%%%%%%%%%%%%%%%%%%%%%%%%%%%%%%%%%%%%%%%%%%%%%%%%%%%%%%%%%%%%%%%%%%%%%
\subsection{Organization and notation}\label{ss:0:ntn}

Let us explain the organization briefly.
In \cref{s:sum}, we give summation formulas related to the function $p(x;q)$ in \cref{dfn:0:p-genq}
which hold for general $q$.
In \cref{s:q}, we study the case where $q$ is a prime power, 
%construct the discrete probability distribution $P_{n,m,k,l;q}$, 
and prove \cref{thm:main} \eqref{i:main:1}.
In \cref{s:q=1}, we study the case $q=1$ and the limit function $p(x)$, 
%construct the discrete probability distribution $P_{n,m,k,l}$, 
and prove \cref{thm:main} \eqref{i:main:2}.
Finally, in \cref{s:q-small}, we study the remaining case $0<\abs{q-1}\ll1$,
and prove \cref{thm:main} \eqref{i:main:3}.

Let us explain the notation used throughout the text.
\begin{enumerate}[nosep]
\item
We denote by $\bbN \ceq \bbZ_{\ge 0} = \{0,1,2,\dotsc\}$ the set of non-negative integers.

\item 
We denote by $\abs{S}$ the number of elements of a finite set $S$.

\item
Notation for numerical functions.
\begin{enumerate}[nosep]
\item 
For $a,b \in \bbZ$, we denote $a \vee b \ceq \max(a,b)$ and $a \wedge b \ceq \min(a,b)$.

\item
For $r \in \bbN$ and a letter $a$, we denote the raising factorial by 
$(a)_r \ceq a (a+1) \dotsm (a+r-1)$. %Note that $(1)_r = r!$.

For commuting letters $a,b,\dotsc,c$, we set 
$(a,b,\dotsc,c)_r \ceq (a)_r (b)_r \dotsm (c)_r$.

\item
For $r \in \bbZ$ and a letter $a$, we denote the binomial coefficient by 
\begin{align*}
 \binom{a}{r} \ceq 
 \begin{cases} a(a-1) \dotsm (a-r+1)/r! & (r \ge 0) \\ %\ff{a}{r}/\ff{r}{r} = 
               0 & (r < 0) \end{cases}.
\end{align*}
%In particular, for $s \in \bbN$ and $r \in \bbZ$ with $r<0$ or $r>s$, 
%we have $\binom{s}{r}=0$.

\item
The generalized hypergeometric series is denoted by 
%with parameters $a_1,\dotsc,a_r$ and $b_1,\dotsc,b_s$ 
\begin{align}\label{eq:ntn:hg}
 \HG{r}{s}{a_1,\dotsc,a_r}{b_1,\dotsc,b_s}{z} := 
 \sum_{i \ge 0} \frac{(a_1,\dotsc,a_r)_i}{(1,b_1,\dotsc,b_s)_i} z^i.
\end{align}
\end{enumerate}

\item
We use bra-ket notation for vector spaces over $\bbC$.
In particular, the paring $\pair{\cdot}{\cdot}$ is anti-linear 
in the first variable (bra), and linear in the second variable (ket).

\item
Let $q$ be a letter and $n \in \bbN$.
\begin{enumerate}[nosep]

\item
For a letter $a$, we set %the $q$-shifted factorial is denoted by 
$(a;q)_n \ceq (1-a) (1-a q) \dotsm (1-a q^{n-1})$. 
%For $m \in \bbZ$ and $n \in \bbN$, we have the following identity.
%\begin{align}\label{eq:qPoch}
% (\q^m;\q)_n = (-1)^n \q^{m n+\binom{n}{2}} (\q^{1-m-n};\q)_n. 
%\end{align}

For commuting letters $a,b,\dotsc,c$, we denote
$(a,b,\dotsc,c;q)_n \ceq (a;q)_n (b;q)_n \dotsm (c;q)_n$.

\item
The basic hypergeometric series is denoted by 
\begin{align}\label{eq:qntn:qhg}
 \qHG{r}{s}{a_1,\dotsc,a_r}{b_1,\dotsc,b_s}{q}{z} :=
 \sum_{i \ge 0} \frac{(a_1,\dotsc,a_r;q)_i}{(q,b_1,\dotsc,b_s;q)_i} 
 \left((-1)^i q^{\binom{i}{2}}\right)^{1+s-r} z^i.
\end{align}

\item
We denote the $q$-integer and the $q$-factorial by 
\begin{align}\label{eq:qntn:qint}
 [n]_q \ceq 1+q+\dotsb+q^{n-1}, \quad [n]_q! \ceq [1]_q [2]_q \dotsm [n]_q,
\end{align}
We abbreviate $[n] \ceq [n]_{q}$ and $[n]! \ceq [n]_q!$ if no confusion may arise.

\item
For $n,m \in \bbN$ with $m \le n$, we denote the $q$-binomial coefficient by 
\begin{align}\label{eq:q:qbinom}
 \qbinom{n}{m}{q} := \frac{(q;q)_n}{(q;q)_m (q;q)_{n-m}}
 = \frac{[n]_{q}!}{[m]_{q}! [n-m]_{q}!}.
\end{align}
In the case $m>n$ or $m<0$, we define $\qbinom{n}{m}{q} \ceq 0$.
We abbreviate $\qbinom{n}{m}{} \ceq \qbinom{n}{m}{q}$ if no confusion may arise.
We will often use the following identities:
\begin{align}\label{eq:qbinom:qinv}
 \qbinom{n}{m}{q} = \frac{(q^n;q^{-1})_m}{(q^m;q^{-1})_m} 
 = \frac{(q^{-n};q)_m}{(q;q)_m} (-q^n)^m q^{-\binom{m}{2}}.
\end{align}
\end{enumerate}
\end{enumerate}

%%%%%%%%%%%%%%%%%%%%%%%%%%%%%%%%%%%%%%%%%%%%%%%%%%%%%%%%%%%%%%%%%%%%%%%%%%%%%%%%%%%%%%%%%%
%%%%%%%%%%%%%%%%%%%%%%%%%%%%%%%%%%%%%%%%%%%%%%%%%%%%%%%%%%%%%%%%%%%%%%%%%%%%%%%%%%%%%%%%%%
\section{Summation formulas}\label{s:sum}

In this section, we prove two summation formulas related to 
the function $p(x;q)$ in \cref{dfn:0:p-genq},
which will be used repeatedly in the sequel.

%%%%%%%%%%%%%%%%%%%%%%%%%%%%%%%%%%%%%%%%%%%%%%%%%%%%%%%%%%%%%%%%%%%%%%%%%%%%%%%%%%%%%%%%%%
\subsection{The first summation formula}\label{ss:sum1}

%For $(n,m,k,l) \in \clN$ and $x \in \{0,1,\dotsc,m\}$, we have 
%\begin{align}
%\begin{split}
%  \wt{p}(x;q) 
%&=q^x \frac{\qbinom{n}{x}{q}}{\qbinom{n}{m}{q}} \frac{[n-2x+1]_q}{[n-x+1]_q} 
%  \sum_{i=0}^x q^{i^2} \qbinom{M}{i}{q} \qbinom{N}{i}{q} \omega(x;i;n,m;q) \\
%&=q^x \frac{\qbinom{n}{x}{q}}{\qbinom{n}{m}{q}} \frac{[n-2x+1]_q}{[n-x+1]_q}
%  \sum_{i=0}^x q^{i^2} \qbinom{M}{i}{q} \qbinom{N}{i}{q} 
%  \qHG{3}{2}{q^{-i},q^{-x},q^{x-n-1}}{q^{-m},q^{m-n}}{q}{q}
%\end{split}
%\end{align}
%with  $M \ceq m-l$ and $N \ceq n-m-k+l$.

\begin{thm}\label{thm:sum1}
For $(n,m,k,l) \in \clN$ and $x\in\{0,1,\dotsc,m\}$, we have the following identity in $\bbQ(q)$.
\begin{align}\label{eq:q:p=qR}
 p(x;q) = q^x \frac{\qbinom{n}{x}{q}}{\qbinom{n}{m}{q}} \frac{[n-2x+1]_q}{[n-x+1]_q}
  \sum_{i=0}^x q^{i^2} \qbinom{M}{i}{q} \qbinom{N}{i}{q} 
  \qHG{3}{2}{q^{-i},q^{-x},q^{x-n-1}}{q^{-m},q^{m-n}}{q}{q},
\end{align}
where $M \ceq m-l$ and $N \ceq n-m-k+l$.
\end{thm}

\begin{proof}
We abbreviate $[n] \ceq [n]_{q}$ and $\qbinom{n}{m}{} \ceq \qbinom{n}{m}{q}$ as before.
It is enough to show the following equality.
\begin{align}\label{eq:q:qR}
\begin{split}
 \sum_{i=0}^m q^{i^2} \qbinom{M}{i}{} \qbinom{N}{i}{} %\omega(x;i;n,m;q) 
&\qHG{3}{2}{q^{-i},q^{-x},q^{x-n-1}}{q^{-m},q^{m-n}}{q}{q} \\
 = \qbinom{n-k}{m-l}{} 
&\qHG{4}{3}{q^{-x},q^{x-n-1},q^{-M},q^{-N}}{q^{-m},q^{m-n},q^{-M-N}}{q}{q}.
\end{split}
\end{align}
Using the formula \eqref{eq:qbinom:qinv} of $q$-binomial coefficient, 
we can compute the left hand side as
\begin{align}\label{eq:q:omega:1}
\begin{split}
  (\text{LHS of \eqref{eq:q:qR}}) 
&=\sum_{i \ge 0} q^{i^2} \qbinom{M}{i}{} \qbinom{N}{i}{} 
  \sum_{r \ge 0} \frac{(q^{-i},q^{-x},q^{x-n-1};q)_r}{(q^{-m},q^{m-n},q;q)_r} q^r \\
&=\sum_{r \ge 0} \frac{(q^{-x},q^{x-n-1};q)_r}{(q^{-m},q^{m-n},q;q)_r} q^r
  \sum_{i \ge r} q^{i^2} (q^{-i};q)_r \frac{(q^M,q^N;q^{-1})_i}{(q,q;q)_i}. 
\end{split}
\end{align}
Using the quality \eqref{eq:qbinom:qinv}, we can rewrite 
the summation over $i$ in terms of $u:=i-r$ as 
\begin{align}
\begin{split}
  \sum_{i \ge r} q^{i^2} (q^{-i};q)_r & \frac{(q^M,q^N;q^{-1})_i}{(q,q;q)_i}
 =\sum_{i \ge r} (-1)^r q^{i^2+\binom{r}{2}-i r} 
  (q^i;q^{-1})_r \frac{(q^M,q^N;q^{-1})_i}{(q,q;q)_i} \\
&=(-1)^r q^{\binom{r}{2}} \sum_{u \ge 0} q^{u^2+u r} 
  (q^{u+r};q^{-1})_r \frac{(q^M,q^N;q^{-1})_{u+r}}{(q,q;q)_{u+r}} \\
&=(-1)^r q^{\binom{r}{2}} \frac{(q^M,q^N;q^{-1})_r}{(q;q)_r}
  \sum_{u \ge 0} \frac{(q^{r-M},q^{r-N};q)_u}{(q,q^{r+1};q)_u} q^{u(M+N-r+1)} \\
&=(-1)^r q^{\binom{r}{2}} \frac{(q^M,q^N;q^{-1})_r}{(q;q)_r} 
  \qHG{2}{1}{q^{r-M},q^{r-N}}{q^{r+1}}{q}{q^{M+N-r+1}}.
\end{split}
\end{align}
Now, let us recall the $q$-Chu-Vandermonde formula
(the terminating case of Heine's $q$-analogue of Gauss's summation formula) 
\cite[(1.5.2)]{GR}: For $a \in \bbN$ we have 
\begin{align}\label{eq:q:qCV}
 \qHG{2}{1}{q^{-a},b}{c}{q}{\frac{c q^a}{b}} = \frac{(c/b;q)_a}{(c;q)_a}.
\end{align}
Applying this formula to the case $(a,b,c)=(M-r,q^{r-N},q^{r+1})$, we have
\begin{align}
 \qHG{2}{1}{q^{r-M},q^{r-N}}{q^{r+1}}{q}{q^{M+N-r+1}} = 
 \frac{(q^{N+1};q)_{M-r}}{(q^{r+1};q)_{M-r}}.
\end{align}
Then, we can continue the computation \eqref{eq:q:omega:1} as 
\begin{align}
\begin{split}
  (\text{LHS of \eqref{eq:q:qR}})
&=\sum_{r \ge 0} \frac{(q^{-x},q^{x-n-1};q)_r}{(q^{-m},q^{m-n},q;q)_r} 
  q^r \cdot (-1)^r q^{\binom{r}{2}} \frac{(q^M,q^N;q^{-1})_r}{(q;q)_r}
  \frac{(q^{N+1};q)_{M-r}}{(q^{r+1};q)_{M-r}} \\
&=\frac{(q;q)_{M+N}}{(q;q)_M (q;q)_N} \sum_{r \ge 0}
  \frac{(q^{-x},q^{x-n-1};q)_r}{(q^{-m},q^{m-n},q;q)_r} \cdot (-1)^r 
  q^{\binom{r}{2}} \frac{(q^M,q^N;q^{-1})_r}{(q^{M+N};q^{-1})_r} q^r \\
&=\qbinom{M+N}{M}{} \sum_{r \ge 0} 
  \frac{(q^{-x},q^{x-n-1},q^{-M},q^{-N};q)_r}{(q^{-m},q^{m-n},q^{-M-N},q;q)_r} 
  q^r = (\text{RHS of \eqref{eq:q:qR}}).
\end{split}
\end{align}
\end{proof}

%%%%%%%%%%%%%%%%%%%%%%%%%%%%%%%%%%%%%%%%%%%%%%%%%%%%%%%%%%%%%%%%%%%%%%%%%%%%%%%%%%%%%%%%%%
\subsection{The second summation formula}\label{ss:sum2}

\begin{thm}\label{thm:q:sum=1}
For $(n,m,k,l) \in \clN$ and $x\in\{0,1,\dotsc,m\}$, we have 
%the summation $s(x;q) \ceq \sum_{u=0}^x p(u;q)$ is 
\begin{align}\label{eq:q:s(y)}
 \sum_{u=0}^x p(u;q) = \qbinom{n-k}{m-l}{q} \frac{\qbinom{n}{x}{q}}{\qbinom{n}{m}{q}} 
 \qHG{4}{3}{q^{-x},q^{x-n},q^{-M},q^{-N}}{q^{-m},q^{m-n},q^{-M-N}}{q}{q}
\end{align}
as an equality in $\bbQ(q)$.
Moreover, in the case $x=m$ we have 
\begin{align}
 \sum_{u=0}^m p(u;q) = 1.
\end{align}
\end{thm}

%As the following proof shows, these are equalities in $\bbQ(q)$.

\begin{proof}
Let us abbreviate $[n] \ceq [n]_{q}$ and $\qbinom{n}{m}{} \ceq \qbinom{n}{m}{q}$. 
Then, 
\begin{align}\label{eq:q:sum=1}
\begin{split}
 s(x;q) 
&= \sum_{u=0}^x p(u;q) 
 =\frac{\qbinom{n-k}{m-l}{}}{\qbinom{n}{m}{}}
  \sum_{u=0}^x q^u \frac{[n-2u+1]}{[n-u+1]} \qbinom{n}{u}{}
  \sum_{r \ge 0} \frac{(q^{-u},q^{u-n-1},q^{-M},q^{-N};q)_r}
                      {(q^{-m},q^{m-n},q^{-M-N},q;q)_r} q^r \\
&=\frac{\qbinom{n-k}{m-l}{}}{\qbinom{n}{m}{}} \sum_{r \ge 0} 
  \frac{(q^{-M},q^{-N};q)_r}{(q^{-m},q^{m-n},q^{-M-N},q;q)_r} q^r
  \sum_{u=r}^x q^u \frac{[n-2u+1]}{[n-u+1]} \qbinom{n}{u}{} (q^{-u},q^{u-n-1};q)_r.
\end{split}
\end{align}
In the inner summation $\sum_{u=r}^x$, we have 
\begin{align}
 q^u \frac{[n-2u+1]}{[n-u+1]} = \frac{1-q^{u-n+r-1}}{1-q^{u-n-1}} - 
 \frac{[u]}{[n-u+1]} \frac{1-q^{-u+r}}{1-q^{-u}},
\end{align}
which yields
\begin{align}
 q^u \frac{[n-2u+1]}{[n-u+1]} \qbinom{n}{u}{} (q^{-u},q^{u-n-1};q)_r
  = f(u;q)-f(u-1;q), \quad 
 f(u;q) \ceq \qbinom{n}{u}{} (q^{-u},q^{u-n};q)_r.
\end{align}
Thus, we have 
\begin{align}
 \sum_{u=r}^x q^u \frac{n-2u+1}{n-u+1} \qbinom{n}{u}{} (q^{-u},q^{u-n-1};q)_r
 = f(x;q)-f(r-1;q) = \qbinom{n}{x}{} (q^{-n+x},q^{-x};q)_r.
\end{align}
Returning to the original summation \eqref{eq:q:sum=1}, we have 
\begin{align}
\begin{split}
  s(x;q) 
&=\frac{\qbinom{n-k}{m-l}{}}{\qbinom{n}{m}{}} \sum_{r \ge 0} 
  \frac{(q^{-M},q^{-N};q)_r}{(q^{-m},q^{m-n},q^{-M-N},q;q)_r} q^r
  \cdot \qbinom{n}{y}{} (q^{-n+x},q^{-x};q)_r \\
&=\qbinom{n-k}{m-l}{} \frac{\qbinom{n}{x}{}}{\qbinom{n}{m}{}} 
  \qHG{4}{3}{q^{-x},q^{x-n},q^{-M},q^{-N}}{q^{-m},q^{m-n},q^{-M-N}}{q}{q}.
\end{split}
\end{align}
In the case $x=m$, we have
\begin{align*}
 s(m;q) 
&=\qbinom{n-k}{m-l}{} 
  \qHG{4}{3}{q^{-m},q^{m-n},q^{-M},q^{-N}}{q^{-m},q^{m-n},q^{-M-N}}{q}{q} 
 =\qbinom{n-k}{m-l}{} \qHG{2}{1}{q^{-M},q^{-N}}{q^{-M-N}}{q}{q} = 1,
\end{align*}
where the last equality follows from \eqref{eq:q:qCV}: 
$\qHG{2}{1}{q^{-M},q^{-N}}{q^{-M-N}}{q}{q}=\qbinom{M+N}{M}{}^{-1}=\qbinom{n-k}{m-l}{}^{-1}$.
\end{proof}

The identity \eqref{eq:q:s(y)} can be rewritten 
as the following basic hypergeometric summation formula concerning
a balanced terminating $\qhg{4}{3}$-series and 
a $0$-balanced terminating $\qhg{3}{2}$-series.
See \cite[p.5]{GR} for the balancing condition,
and \cite{AAR,GR,S} for the hypergeometric summation formulas.
Our equality seems to be new.

\begin{cor}\label{cor:q:sum}
Using $M \ceq m-l$ and $N \ceq n-m-k+l$, we have the following identity in $\bbQ(q)$.
\begin{align}
\begin{split}
 \sum_{u=0}^x q^u \qbinom{n}{u}{q} \frac{[n-2u+1]_q}{[n-u+1]_q}   
&\qHG{4}{3}{q^{-u},q^{u-n-1},q^{-M},q^{-N}}{q^{-m},q^{m-n},q^{-M-N}}{q}{q} \\
=\qbinom{n}{x}{q} 
&\qHG{4}{3}{q^{-x},q^{x-n},q^{-M},q^{-N}}{q^{-m},q^{m-n},q^{-M-N}}{q}{q}.
\end{split}
\end{align}
\end{cor}

%%%%%%%%%%%%%%%%%%%%%%%%%%%%%%%%%%%%%%%%%%%%%%%%%%%%%%%%%%%%%%%%%%%%%%%%%%%%%%%%%%%%%%%%%%
%%%%%%%%%%%%%%%%%%%%%%%%%%%%%%%%%%%%%%%%%%%%%%%%%%%%%%%%%%%%%%%%%%%%%%%%%%%%%%%%%%%%%%%%%%
\section{The case $q$ is a prime power}
\label{s:q}

The purpose of this section is to show \cref{thm:main} \eqref{i:main:1}
in the case where $q$ is a prime power, using the representation theory of 
the general linear group $\GL(n,\Fq)$ over a finite field $\Fq$ of order $q$.

Let us explain the outline of the proof.
Recall the parameter set $\clN$ in \eqref{eq:0:clN}, and let  $(n,m,k,l) \in \clN$.
In \cref{ss:q:pos}, we introduce a permutation representation $T$ of $\GL(n,\Fq)$ 
(see \eqref{eq:q:T} and \eqref{eq:q:varphi}), 
and construct a new quantity $\wt{p}(x;q)$
for $x \in \{0,1,\dotsc,m\}$ (see \eqref{eq:q:p}).
We have $\wt{p}(x;q) \ge 0$ by \eqref{eq:q:pos}.
Then, in \cref{thm:q:qH}, we show that $\wt{p}(x;q)$ is equal to 
the right hand side of the identity \eqref{eq:q:p=qR} on $p(x;q)$.
Hence, we have the positivity $p(x;q) \ge 0$. 
Since $\sum_{x=0}^m p(x;q)=1$ by \cref{thm:q:sum=1}, 
we find that $p(x;q)$ is the pmf of a discrete probability distribution on $\{0,1,\dotsc,m\}$.
The remaining statement of \cref{thm:main} \eqref{i:main:1} 
on the cdf is already shown in \cref{thm:q:sum=1}.

In this section, we choose and fix a finite field $\Fq$ of order $q$.

%%%%%%%%%%%%%%%%%%%%%%%%%%%%%%%%%%%%%%%%%%%%%%%%%%%%%%%%%%%%%%%%%%%%%%%%%%%%%%%%%%%%%%%%%%
%%%%%%%%%%%%%%%%%%%%%%%%%%%%%%%%%%%%%%%%%%%%%%%%%%%%%%%%%%%%%%%%%%%%%%%%%%%%%%%%%%%%%%%%%%
\subsection{Gelfand pair \texorpdfstring{$(\GL(n,\Fq),P(m,n-m,\Fq))$}{(GL(n,Fq),P(m,n-m,Fq)} 
 and \texorpdfstring{$q$}{q}-Hahn polynomials}
\label{ss:q:sph}

Here we explain some preliminary facts from the representation theory of 
the general linear group $\GL(n,\Fq)$ over the finite field $\Fq$.
Specifically, we consider a Gelfand pair of $\GL(n,\Fq)$ and its parabolic subgroup.
By the general theory of Gelfand pair (see \cite[VII.1]{M} and \cite[II.11]{BI} for example),
we have the associated zonal spherical functions.
In the next \cref{ss:q:pos}, we will use these functions to 
establish the equality between our function $p(x;q)$ 
and a positive quantity $\wt{p}(x;q)$ arising from a certain $\GL(n,\Fq)$-representation, 
and to prove the positivity $p(x;q) \ge 0$. 
The symbols used below are borrowed from \cite[VII.1, Exercise 14]{M}.

Let $\clV$ be a vector space of dimension $n$ over the finite field $\Fq$ of order $q$,
and $\clU \subset \clV$ be a subspace of dimension $m$ with $m \le n-m$.
We consider the following pair of finite groups:
\begin{align}
 G \ceq \GL(\clV) \, \supset \, K \ceq \{ g \in G \mid g \, \clU = \clU\}.
\end{align}
The group $G$ is isomorphic to the general linear group $\GL(n,\Fq)$ over $\Fq$,
and $K$ is isomorphic to the maximal parabolic subgroup 
$P(m,n-m,\Fq) \subset \GL(n,\Fq)$ consisting of block-upper-triangular matrices
$\begin{sbm}* & * \\ 0 & *\end{sbm}$ with block sizes $m$ and $n-m$.
We denote by $1_K$ the trivial representation of $K$, 
and by $1_K^G \ceq \Ind_K^G 1_K = \bbC[G] \otimes_{\bbC[K]}1_K$ 
the induced representation with coefficients in $\bbC$.
The latter is isomorphic to the permutation representation of $G$ on the Grassmannian 
\begin{align}
 G/K \cong \Gr(m,n) \ceq \{\clW \subset \clV \mid \text{linear subspaces}, \, \dim_{\Fq} \clW = m\}.
\end{align}
Recall that the number of elements is given by the $q$-binomial coefficient \eqref{eq:q:qbinom}:
\begin{align}\label{eq:q:G/K}
 \abs{G/K} = \frac{\abs{G}}{\abs{K}} = \abs{\Gr(m,n)} = \qbinom{n}{m}{q}.
\end{align}

The pair $(G,K)$ is a Gelfand pair by \cite[VII.1, Exercise 14]{M}.
In other words, the induced representation $1_K^G$ is multiplicity-free.
According to loc.\ cit., the irreducible decomposition is given by
\begin{align}\label{eq:q:1_K^G}
 1_K ^G = \sum_{x=0}^m V_x, %\quad \ch(\chi_{V_x}) = s_{(n-x,x)}(Y_1),
\end{align}
and the irreducible $G$-representation ($G$-irrep for short) $V_x$ 
%is of labeled by the Young diagram $(n-x,x)$, 
has the dimension
\begin{align}\label{eq:q:dimV_x}
 \dim_{\bbC} V_x 
&= \frac{q^x \, [n]_q!}{\prod_{\square \in (n-x,x)} [h(\square)]_q}
 = q^x \frac{[n-2x+1]_q}{[n-x+1]_q} \qbinom{n}{x}{q}.
\end{align}
Here the index $\square$ of the product runs over the boxes in the Young diagram $(n-x,x)$, 
and $h(\square)$ denotes the hook length of the box $x$.
We will not give an explicit definition of the representation $V_x$.
Let us only mention the character formula of $V_x$:
It has an explicit formula using the Schur symmetric function $s_{(n-x,x)}(y)$
associated to the Young diagram $(n-x,x)$ with a certain infinite family of variables $y$.
See \cite[VII.1, Exercise 14]{M} for the detail 
and \cite[Chap.\ IV]{M} for the theory of irreducible characters of $G=\GL(n,\Fq)$.
In the discussion below, we only use the fact that the $G$-irreps $V_0,\dotsc,V_m$ 
appear in the decomposition \eqref{eq:q:1_K^G}, 
and that they have the dimension \eqref{eq:q:dimV_x}.
%See \cite[\S1.1]{M} for the notation.
%The map $\ch$ in \eqref{eq:q:1_K^G} is the characteristic map \cite[IV.4]{M}, 
%a bijection sending the character $\chi_{V_x}$ to a symmetric function of 
%the infinite family of variables $Y=\{y_{i,f} \mid i=1,2,\dotsc, \, f \in \Phi\}$.
%See \cite[IV.1]{M} for the definition of the set $\Phi$.
%The symbol $s_{(n-x,x)}(Y_1)$ denotes the Schur symmetric function associated to 
%the partition $(n-x,x)$ with variables $Y_1 := \{y_{i,f=1} \mid i=1,2,\dotsc\} \subset Y$.

By \cite[VII.1, (1.1)]{M}, the multiplicity-freeness of $1_K^G$ is equivalent to 
the commutativity of the Hecke algebra $\clC(G,K)$, which is the vector space
$\clL(K \bs G /K)$ of $K$-bi-invariant functions on $G$
equipped with the convolution product $*$.
For the irrep $V_x$ 
appearing in the decomposition \eqref{eq:q:1_K^G},
the zonal spherical function $\omega_x$ associated to $V_x$ is defined as 
\begin{align}
 \omega_x \ceq \ol{\chi}_x * e_K,
\end{align}
where $\ol{\chi}_x \ceq \ol{\chi}_{V_x}$ is 
the complex conjugate of the character of $V_x$,
and $e_K \ceq \frac{1}{\abs{K}}\sum_{k \in K}k$ is the idempotent associated to the subgroup $K$.
%\end{align}
We also have $\omega_x \in \clC(G,K) \cap \clV_x$ by \cite[VII.1, (1.4)]{M}.

For $i=0,1,\dotsc,m$, we set 
\begin{align}\label{eq:q:K_i}
 K_i \ceq \{ g \in G \mid \dim_{\Fq}(\clU \cap g \, \clU) = m-i\}.
\end{align}
In particular, we have $K_0=K$.
These subsets $K_i \subset G$ form a complete system of representatives for 
the double cosets of $K$ in $G$, i.e., $K \bs G / K = \{K_0,K_1,\dotsc,K_m\}$.
By \cite[VII.1, Exercise 14]{M}, the number $\abs{K_i}/\abs{K}$ is given as 
\begin{align}\label{eq:q:K_i/K}
 \frac{\abs{K_i}}{\abs{K}} = q^{i^2} \qbinom{m}{i}{q} \qbinom{n-m}{i}{q}.
\end{align}
Let us cite the proof from loc.\ cit.\ for later use.
It is enough to count the number of $m$-dimensional subspaces $\clW \subset \clV$ 
such that $\dim_{\Fq}(\clU \cap \clW)=m-i$.
Let $\clE \ceq \clU \cap \clW$ and $\clF \ceq \clU+\clW$, 
so that $\clE$ is an $(m-i)$-dimensional subspace of $\clU$
and $\clF$ is an $(m+i)$-dimensional subspace of $\clV$.
Since $\clW/\clE$ is the algebraic complement of $\clU/\clE$ in $\clF/\clE$, 
a choice for $\clW$ corresponds to a choice of $(\clE,\clF,f)$, 
where $f \colon \clW/\clE \to \clU/\clE$ is a linear map.
Since $\dim_{\Fq}(\clW/\clE)=\dim_{\Fq}(\clU/\clE)=i$, we have \eqref{eq:q:K_i/K}.

We are now interested in the value of the zonal spherical function $\omega_x$ 
at elements in $K_i$, which is denoted by
\begin{align}
 %\omega(x;i;q) = 
 \omega(x;i;n,m;q) \ceq \omega_x(g), 
 \quad g \in K_i \quad (x,i = 0,1,\dotsc,m).
\end{align}
We have $\omega(x;0;q)=\omega_x(e)=1$ by \cite[VII.1, (1.4)]{M}.

An explicit formula for the value $\omega(x;i;q)$ is given by Delsarte \cite{De1,De2}.
We explain it below, following the presentation by Dunkl \cite{Du2} with some modification.
Let $Q_d$ be the $q$-Hahn polynomial, which is an orthogonal polynomial 
defined by the basic hypergeometric series $\qhg{3}{2}$ as 
\begin{align}\label{eq:q:qHahn}
 Q_d(q^{-z};\alpha,\beta,D;q) \ceq 
 \qHG{3}{2}{q^{-z},q^{-d},\alpha \beta q^{d+1}}{\alpha q,q^{-D}}{q}{q} \quad
 (D \in \bbN, \ d=0,1,\dotsc,D).
\end{align}
We followed \cite[\S 14.6]{KLS} for the notation.
Then, for $m,n \in \bbZ$ with $0 \le m \le n-m$ and for $x,i=0,1.\dotsc,m$, we have
\begin{align}\label{eq:q:omega}
 \omega(x;i;n,m;q) = Q_x(q^{-i};q^{-m-1},q^{m-n-1},q^{n-m};q) 
 = \qHG{3}{2}{q^{-i},q^{-x},q^{x-n-1}}{q^{-m},q^{m-n}}{q}{q}.
\end{align}

\begin{rmk}
Two remarks on the formula \eqref{eq:q:omega} are in order.
\begin{enumerate}[nosep]
\item
The formula is a $q$-analogue of the formula %in \cref{sss:Hora:H}
for the value $\omega_{(n-x,x)}(i)$ of the zonal spherical function on the Gelfand pair 
$(\frS_n,\frS_m \times \frS_{n-m})$.
See \cite[III.2]{BI} and \cite[Remark 4.1.3 (2)]{HHY} for the detail.

\item 
In \cite[3.7 Theorem]{Du2}, the formula is given in terms of the $q^{-1}$-Hahn polynomial:
\begin{align}
 \omega(x;i;n,m;q) = 
 (-1)^x q^{\binom{x}{2}-m x} \frac{(q^m;q^{-1})_x}{(q^{n-m};q^{-1})_x}
 \qHG{3}{2}{q^{m-i},q^x,q^{n-m+1}}{q^m,q^m}{q^{-1}}{q^{-1}}.
\end{align}
We can show that it is equivalent to the $q$-Hahn polynomial formula \eqref{eq:q:omega}
with the help of Sears' $\qhg{3}{2}$ transformation formula \cite[\S 3.2]{GR} 
and Heine's $\qhg{2}{1}$ summation formula 
($q$-analogue of Gauss' $\hg{2}{1}$ summation formula) \cite[\S 1.5]{GR}.
The computation is left to the reader. 
\end{enumerate}
\end{rmk}

%%%%%%%%%%%%%%%%%%%%%%%%%%%%%%%%%%%%%%%%%%%%%%%%%%%%%%%%%%%%%%%%%%%%%%%%%%%%%%%%%%%%%%%%%%
\subsection{The positivity from \texorpdfstring{$\GL(n,\Fq)$}{GL(n,Fq)}-action} 
%distribution \texorpdfstring{$P_{n,m,k,l;q}$}{Pn,m,k,l;q}}
\label{ss:q:pos}

As mentioned in the beginning of this \cref{s:q}, 
we prove the positivity of the function $p(x;q)$ using a representation of 
the general linear group $\GL(n,\Fq)$.
The representation space will be the tensor space $T=(\bbC^2)^{\otimes [n]_q}$,
where $[n]_q = 1+q+\dotsb+q^{n-1}$ is the $q$-integer.
Let us start with the definition of this representation.

Let $(n,m,k,l)$ be the parameters belonging to the set $\clN$ in \eqref{eq:0:clN}, 
and $\clV$ be an $n$-dimensional $\Fq$-linear space as in \cref{ss:q:sph}.
We denote by $\bbP^1(\clV)$ the set of lines, or one-dimensional $\Fq$-linear subspaces, in $\clV$.
Note that $\bbP^1(\clV)$ is a finite set with cardinality $\abs{\bbP^1(\clV)}=[n]_q$.

Let $\bbC^2 = \bbC \ket{0} \oplus \bbC \ket{1}$ be the two-dimensional complex linear space
with basis $\ket{0}$ and $\ket{1}$, using bra-ket notation. Consider the tensor space
\begin{align}\label{eq:q:T}
 T \ceq (\bbC^2)^{\otimes \bbP^1(\clV)} = (\bbC^2)^{\otimes [n]_q}.
\end{align}
Here $\bbP^1(\clV)$ is the index set of tensors, and an element $\ket{t} \in T$ will be 
expressed as $\ket{t} = \otimes_{L \in \bbP^1(\clV)} t^{(L)}$.

We introduce a linear basis of $T$ in the following way.
For each linear subspace $\clW \subset \clV$, we define $\ket{e_{\clW}} \in T$ by 
\begin{align}\label{eq:q:eW-dfn}
 \ket{e_{\clW}} = \otimes_{L \in \bbP^1(\clV)} e_{\clW}^{(L)}, \quad 
 e_{\clW}^{(L)} \ceq
 \begin{cases} \ket{1} & (L \subset \clW) \\ \ket{0} & (L \not\subset \clW)\end{cases}.
\end{align}
For example, we have $\ket{e_{\clV}}=\ket{1}^{\otimes \bbP^1(\clV)}$ and 
$\ket{e_O} = \ket{0}^{\otimes \bbP^1(\clV)}$, where $O = \{0\} \subset \clV$ is the zero space.
Then, for linear subspaces $\clW, \clW' \subset \clV$,  we have 
\begin{align}\label{eq:q:e_W}
 \ket{e_{\clW}} = \ket{e_{\clW'}} \iff \bbP^1(\clW) = \bbP^1(\clW') \iff \clW = \clW'.
\end{align}
Here $\bbP^1(\clW)$ denotes the set of lines in $\clW$, 
and we regard it as a subset of $\bbP^1(\clV)$ by the inclusion $\bbP^1(\clW) \inj \bbP^1(\clV)$
induced by the given inclusion $\clW \inj \clV$.
By \eqref{eq:q:e_W}, we see that 
the elements $\ket{e_{\clW}}$ for all subspaces $\clW \subset \clV$ form a linear basis of $T$.

Hereafter we denote $G \ceq \GL(\clV) = \GL(n,\Fq)$.
The natural action of $G$ on $\clV$ induces another one on $\bbP^1(\clV)$,
and it further induces a $G$-action $\varphi$ on $T$ by permutation of tensor factors.
Explicitly, we have 
\begin{align}\label{eq:q:varphi}
 \varphi(g) \ket{e_{\clW}} = \ket{e_{g \clW}} \quad (g \in G).
\end{align}
Thus, we regard $G=\GL(n,\Fq)$ as a subgroup of the symmetric group $\frS_{[n]_q}$ 
which permutes the tensor factor of $T=(\bbC^2)^{\otimes [n]_q}$.

The standard Hermitian form on $\bbC^2$ induces another one 
on the tensor space $T=(\bbC^2)^{\otimes [n]_q}$.
We denote this pairing on $T$ by bra-ket notation.
Then, \eqref{eq:q:e_W} implies
\begin{align}\label{eq:q:pair}
 \pair{e_{\clW}}{e_{\clW'}} = \delta_{\clW,\clW'}, 
\end{align}
and $\ket{e_{\clW}}$ for all $\clW \subset \clV$ form an orthonormal basis 
of the Hermitian space $T$.
Also, by \eqref{eq:q:e_W} and \eqref{eq:q:varphi}, we see that 
the $G$-representation $T$ is unitary with respect to this Hermitian form.

Finishing the introduction of the $G$-representation $T$, 
we turn to the construction of quantities $\wt{p}(x;q) \ge 0$ for $x \in \{0,1,\dotsc,m\}$.
The construction is divided into three steps.

As the first step, we consider the irreducible decomposition of $T$ as a $G$-module:
$T = \bigoplus_{\lambda \in \wh{G}} T_\lambda$, 
where $T_\lambda$ is the isotypical component 
corresponding to the $G$-irrep $V_\lambda$.
That is, we have $T_\lambda = V_\lambda^{\oplus m_\lambda}$, 
where $m_\lambda$ denotes the multiplicity of $V_\lambda$ in $T$.
Now the orthogonal projector $\sfP_\lambda\colon T \srj T_\lambda$ 
for each $\lambda \in \wh{G}$ is expressed as 
\begin{align}\label{eq:q:P_a}
 \sfP_{\lambda;q} = \frac{\dim_{\bbC} V_\lambda}{\abs{G}} 
 \sum_{g \in G} \ol{\chi_\lambda(g)} \varphi(g),
\end{align}
where $\chi_\lambda \ceq \chi_{V_\lambda}$ is the character of $V_\lambda$,
and the overline denotes the complex conjugate.
This formula is an outcome of the fact that 
the element $(\dim T_\lambda/\abs{G}) \ol{\chi_\lambda}$
is a primitive central idempotent in the group algebra $\bbC[G]$.
See \cite[V.33, (33.8)]{CR} for the detail.
%Since $\chi_x$ is $\bbR$-valued, we suppress the complex conjugate in \eqref{eq:q:P_a} hereafter.

The projector $\sfP_{\lambda,q}$ is self-adjoint for the Hermitian form on $T$.
In fact, as explained at \eqref{eq:q:varphi}, the action $\varphi$ permutes 
the elements $\ket{e_{\clW}}$ of the orthogonal basis of $T$.
Hence, the adjoint of $\sum_{g \in G} \ol{\chi_\lambda(g)} \varphi(g^{-1})$ is 
$\sum_{g \in G}\chi_\lambda(g^{-1}) \varphi(g)=\sum_{g \in G}\ol{\chi_\lambda(g)}\varphi(g^{-1})$,
and we have $\sfP_{\lambda,q}^\dagger = \sfP_{\lambda,q}$.
Now, since $\sfP_{\lambda,q}$ is an idempotent, we have for any element $\ket{\psi} \in T$ that 
\begin{align}\label{eq:q:gen-pos}
 \bra{\psi} P_{\lambda,q} \ket{\psi} = \bra{\psi} P_{\lambda,q}^2 \ket{\psi} = 
 (\bra{\psi}P_{\lambda,q}^\dagger) (P_{\lambda,q}\ket{\psi}) \ge 0.
\end{align}

\begin{rmk}[added after the referees' comments]\label{rmk:ref2}
The positivity \eqref{eq:q:gen-pos} is indicated 
by the referees' suggestions in the revise process.
\end{rmk}

As the second step, we choose a distinguished element $\ket{\Xi_{n,m|k,l;q}} \in T$ 
as the above $\ket{\psi}$.
Fix a linear basis $v_1, v_2, \dotsc, v_n$ of $\clV$. 
For each interval $[i,j]=\{i,i+1,\dotsc,j\} \subset \{1,2,\dotsc,n\}$, 
we set $\clV_{[i,j]} \ceq \Fq v_i + \Fq v_{i+1} + \dotsb + \Fq v_j$.
We also denote $\clV_i \ceq \clV_{[1,i]}$.
In particular, we have $\clV=\clV_n=\clV_{[1,n]}$, and $\clV$ is equipped with 
the complete flag $\clV \supset \clV_{n-1} \supset \dotsb \supset \clV_1 \supset \{0\}$.
Then, the subspace 
\begin{align}\label{eq:q:cLV_k^c}
 \clV_k^c \ceq \clV_{[k+1,n]}
\end{align}
is a complement of $\clV_k \subset \clV$.
Now we define an element $\ket{\Xi_{n,m | k,l;q}} \in T$ by
\begin{align}\label{eq:q:Xi}
&\ket{\Xi_{n,m | k,l;q}} \ceq \qbinom{n-k}{m-l}{q}^{-1/2} 
 \sum_{\clW \in \Gr(n,m \midd k,l)} \ket{e_{\clW}}, \\
%\nonumber
&\Gr(n,m \midd k,l) \ceq \{ \clW \mid \clV_l \subset \clW \subset \clV, \, 
 \dim_{\Fq} \clW = m, \, \dim_{\Fq}(\clW \cap \clV_k^c)=m-l \}.
\end{align}
The $q$-binomial in \eqref{eq:q:Xi} is the normalization factor.
To see that, consider the correspondence
$\clW \mapsto \clW' \ceq \clW \cap \clV_k^c = \clW \cap \clV_{[k+1,n]}$, 
which induces a bijection
\begin{align}
 \Gr(n,m \midd k,l) \lto 
 \{ \clW' \subset \clV_{[k+1,n]} \mid \dim_{\Fq} \clW' = m-l \},
\end{align}
and the right hand side is isomorphic to the Grassmannian $\Gr(m-l,n-k)$ over $\Fq$.
By \eqref{eq:q:G/K}, we have 
\begin{align}%\label{eq:q:Xi-pair}
 \pair{\Xi_{n,m | k,l;q}}{\Xi_{n,m | k,l;q}} = 
 \qbinom{n-k}{m-l}{q}^{-1} \, \abs{\Gr(m-l,n-k)} = 1.
\end{align}

Now we are in the final step.
Recall the $G$-irreducible decomposition $1_K^G = \sum_{x=0}^m V_x$ of the induced module $1_K^G$ 
for the parabolic subgroup $K=P(m,n-m,\Fq)$ in \eqref{eq:q:Xi}.
We denote by $\sfP_{x;q}$ the projector $\sfP_{\lambda;q}$ \eqref{eq:q:P_a} 
for $\lambda \in \wh{G}$ corresponding to the $G$-irrep $V_x$, and define
\begin{align}\label{eq:q:p}
 \wt{p}(x;q) = \wt{p}(x \midd n,m,k,l;q) \ceq
 \bra{\Xi_{n,m|k,l;q}} \sfP_{x;q} \ket{\Xi_{n,m|k,l;q}}
\end{align}
for $x \in \{0,1,\dotsc,m\}$.
By \eqref{eq:q:gen-pos}, we have 
\begin{align}\label{eq:q:pos}
 \wt{p}(x;q) \ge 0.
\end{align}

%The replacement $\ket{0} \leftrightarrow \ket{1}$ of the basis elements of $\bbC^2$ 
%yields the symmetry
%\begin{align}\label{eq:q:sym}
% \wt{p}(x \midd n,m,k,l;q) = \wt{p}(x \midd n,n-m,k,k-l;q).
%\end{align}
%Hence, we assume $m \le n-m$ and consider the restricted parameter set $\clN$ in \cref{eq:0:clN}.
%The range of $x$ is then $\{0,1,\dotsc,m\}$.

The main result of this section is:

\begin{thm}[$q$-Hahn presentation]\label{thm:q:qH}
For $(n,m,k,l) \in \clN$ and $x \in \{0,1,\dotsc,m\}$, we have 
\begin{align}
\begin{split}
  \wt{p}(x;q) 
&=q^x \frac{\qbinom{n}{x}{q}}{\qbinom{n}{m}{q}} \frac{[n-2x+1]_q}{[n-x+1]_q} 
  \sum_{i=0}^x q^{i^2} \qbinom{M}{i}{q} \qbinom{N}{i}{q} \omega(x;i;n,m;q) \\
&=q^x \frac{\qbinom{n}{x}{q}}{\qbinom{n}{m}{q}} \frac{[n-2x+1]_q}{[n-x+1]_q}
  \sum_{i=0}^x q^{i^2} \qbinom{M}{i}{q} \qbinom{N}{i}{q} 
  \qHG{3}{2}{q^{-i},q^{-x},q^{x-n-1}}{q^{-m},q^{m-n}}{q}{q}
\end{split}
\end{align}
with  $M \ceq m-l$ and $N \ceq n-m-k+l$.
\end{thm}

%Note that the condition $(n,m,k,l) \in \clN$ implies $m \le n-m$ and $M,N \ge 0$.

\begin{proof}
In this proof, we abbreviate 
$[n] \ceq [n]_{q}$ and $\qbinom{n}{m}{} \ceq \qbinom{n}{m}{q}$.
Denote $G \ceq \GL(\clV) \cong \GL(n,\bbF_q)$.
By \eqref{eq:q:Xi}, \eqref{eq:q:P_a} and \eqref{eq:q:pair}, we have 
\begin{align}
 \wt{p}(x;q) 
&=\bra{\Xi_{n,m|k,l;q}} \sfP_{x;q} \ket{\Xi_{n,m|k,l;q}} 
\notag \\
&=\frac{\dim_{\bbC} V_x}{\abs{G}} \qbinom{n-k}{m-l}{}^{-1} \sum_{g \in G} \chi_x(g)
  \sum_{\clW,\clW' \in \Gr(n,m|k,l)} \bra{e_{\clW}} \varphi(g) \ket{e_{\clW'}} 
\notag \\
&=\frac{\dim_{\bbC} V_x}{\abs{G}} \qbinom{n-k}{m-l}{}^{-1} 
  \sum_{\clW,\clW' \in \Gr(n,m|k,l)} \sum_{g \in G} \chi_x(g) \delta_{\clW, g \clW'}.
\label{eq:q:XPX1}
\end{align}
For a while, we fix $\clW,\clW' \in \Gr(n,m|k,l)$, and study the elements $g \in G$
which are effective in the summation.
Take $g_0 \in G$ satisfying the following conditions:
\begin{align}\label{eq:q:g_0}
 \rst{g_0}{\clV_k} = \id, \quad \rst{g_0}{\clW'} \colon \clW' \lsto \clW, \quad 
 \rst{g_0}{{\clW'}^c \cap \clV_{[k+1,n]}} \colon 
 {\clW'}^c \cap \clV_{[k+1,n]} \lsto \clW^c \cap \clV_{[k+1,n]}.
\end{align}
We also set 
\begin{align}
 P(\clW) \ceq \{g \in G \mid g \clW = \clW\},
\end{align}
which is a maximal parabolic subgroup of $G$ isomorphic to $P(m,n-m,\Fq)$.
Then, we find that 
\begin{align}
 \delta_{\clW,g \clW'} \neq 0 \iff g g_0^{-1} \in P(\clW).
\end{align}
Next, recalling that we choose $\clU=\clV_m$, we take $\sigma_0 \in G$ 
satisfying the following conditions:
\begin{align}\label{eq:q:sigma_0}
\begin{split}
&\rst{\sigma_0}{\clV_l} = \id, \quad \rst{\sigma_0}{\clU} \colon \clU \lsto \clW, \quad
 \sigma_0(v_{m+j}) = v_j \ (j=1,\dotsc,k-l), \\
&\rst{\sigma_0}{\clV_{[m+k-l+1,n]}} \colon 
 \clV_{[m+k-l+1,n]} \lsto \clW^c \cap \clV_{[k+1,n]}.
\end{split}
\end{align}
Here we used the fixed basis $v_1,\dotsc,v_n$ of $\clV$.
Then, we have 
\begin{align}
 \delta_{\clW,g \clW'} \neq 0 \iff \sigma_0^{-1} g g_0^{-1} \sigma_0 
 \in \sigma_0^{-1} P(\clW) \sigma_0 = P(\clU) = K,
\end{align}
and \eqref{eq:q:XPX1} is continued as 
\begin{align}\label{eq:q:XPX2}
%\bra{\Xi_{n,m|k,l;q}} \sfP_{x;q} \ket{\Xi_{n,m|k,l;q}} 
 \wt{p}(x;q)
&= \frac{\dim_{\bbC} V_x}{\abs{G}} \qbinom{n-k}{m-l}{}^{-1} 
 \sum_{\clW,\clW' \in \Gr(n,m|k,l)} 
 \sum_{k \in K} \chi_x(\sigma_0 k \sigma_0^{-1} g_0).
\end{align}
Thus, we want to calculate 
$\chi_x(\sigma_0 k \sigma_0^{-1} g_0)
=\chi_x\bigl(k (\sigma_0^{-1}g_0^{-1}\sigma_0)^{-1}\bigr)$.
We set 
\begin{align}
 s \ceq \dim_{\Fq}(\clW \cap \clW' \cap \clV_l^c), \quad 
 i \ceq m-l-s = m - \dim_{\Fq}(\clW \cap \clW').
\end{align}
Then, noting that $\clU \cap \clV_l^c=\clV_{[l+1,m]}$, we can find from the conditions
\eqref{eq:q:g_0} and \eqref{eq:q:sigma_0} that the element 
$\sigma \ceq \sigma_0^{-1}g_0^{-1}\sigma_0$ satisfies
\begin{align}
 \dim_{\Fq} \bigl(\sigma(\clV_{[l+1,m]}) \cap \clV_{[l+1,m]}\bigr) = s, \quad
 \rst{\sigma}{\clV_l} = \id, \quad \rst{\sigma}{\clV_{[m+1,m+k-l]}} = \id.
\end{align}
Hence, there exist some $k',k'' \in K=P(\clU)$ such that 
\begin{align}
 \sigma = k' h k'', \quad 
 h \ceq g(l+1,m+k-l+1) \, g(l+2,m+k-l+2) \dotsm g(l+i,m+k-l+i),
\end{align}
where $g(l+j,m+k-l+j) \in G$ is the reflection mapping 
the basis vector $v_{l+j}$ to $v_{m+k-l+j}$ and vice-a-versa.
The element $h$ is independent of $\clW,\clW'$, 
and belongs to $K_i$ given in \eqref{eq:q:K_i}.
Then, using the formula \eqref{eq:q:omega} we have 
\begin{align}
 \sum_{k \in K} \chi_x(\sigma_0 k \sigma_0^{-1} g_0) = 
 \sum_{k \in K} \chi_x(k {k''}^{-1} h {k'}^{-1}) = 
 \sum_{k \in K} \chi_x(k h) = \abs{K} \omega(x;i;n,m;q).
\end{align}
The computation \eqref{eq:q:XPX2} is now continued as 
\begin{align}
 \wt{p}(x;q) = 
 \frac{\dim_{\bbC} V_x}{\abs{G}} \qbinom{n-k}{m-l}{}^{-1} \sum_{i=0}^x \, 
 \sum_{\substack{\clW,\clW' \in \Gr(n,m|k,l), \\ m-\dim(W \cap W')=i}} 
 \abs{K} \omega(x;i;n,m;q).
\end{align}
Note that the final term is independent of $\clW$ and $\clW'$.
Then, a similar argument as in \eqref{eq:q:K_i/K} yields
\begin{align}
\begin{split}
&\abs{\{(\clW,\clW') \in \Gr(n,m \midd k,l)^2 \mid m-\dim_{\Fq}(\clW \cap \clW')=i\}} \\
&= \abs{\Gr(n,m|k,l)} \cdot q^{i^2} \qbinom{m-l}{i}{} \qbinom{n-m-k+l}{i}{}
 = \qbinom{n-k}{m-l}{} q^{i^2} \qbinom{M}{i}{} \qbinom{N}{i}{}
\end{split}
\end{align}
with $M \ceq m-l$ and $N \ceq n-m-k+l$.
Then, \eqref{eq:q:G/K} and the dimension formula \eqref{eq:q:dimV_x} yield
\begin{align}
\begin{split}
 \wt{p}(x;q) %\pair{\Xi_{n,m|k,l;q}}{\sfP_{x;q} \Xi_{n,m|k,l;q}} 
&= \frac{\abs{K}}{\abs{G}} \dim_{\bbC} V_x 
 \sum_{i=0}^x q^{i^2} \qbinom{M}{i}{} \qbinom{N}{i}{} \omega(x;i;n,m;q) \\
&= \qbinom{n}{m}{}^{-1} q^x \frac{[n-2x+1]}{[n-x+1]} \qbinom{n}{x}{}
 \sum_{i=0}^x q^{i^2} \qbinom{M}{i}{} \qbinom{N}{i}{} \omega(x;i;n,m;q).
\end{split}
\end{align}
\end{proof}

Now we prove \cref{thm:main} \eqref{i:main:1}.
By \cref{thm:sum1} and \cref{thm:q:qH}, we have $\wt{p}(x;q) = p(x;q)$.
Since $\wt{p}(x;q) \ge 0$ by \eqref{eq:q:pos}, we have $p(x;q) \ge 0$.
Then, since $\sum_{x=0}^m p(x;q)=1$ by \cref{thm:q:sum=1}, 
the function $p(x;q)$ is the pmf of a discrete probability distribution on $\{0,1,\dotsc,m\}$.
The cdf is given by the $\qhg{4}{3}$-series in \cref{thm:q:sum=1}. 
The proof of \cref{thm:main} \eqref{i:main:1} is now finished.

\begin{rmk}[added after referee's comments]\label{rmk:ref3}
The positivity $p(x;q) \ge 0$ implies the following remarks.
\begin{enumerate}
\item 
The $q$-Racah parameters in the function $p(x;q)$ lie outside the region of orthogonality. 
In fact, a non-constant polynomial in a family of orthogonal polynomials must 
take on some negative values on a finite point-set domain of orthogonality, 
but we have $p(x;q) \ge 0$ for $x\in\{0,1,\dotsc,m\}$.

\item
For a fixed $(x,n,m,k,l)$, the value $p(x;q)$ is rational in $q$
and positive on the infinite sequence $q \to \infty$.
Hence, it must be eventually positive 
since it can have only a finite number of sign-changes.
Therefore, $p(x;q)$ is the pmf of a discrete probability distribution on $\{0,1,\dotsc,m\}$
for sufficiently large $q$.
\end{enumerate}
\end{rmk}

%%%%%%%%%%%%%%%%%%%%%%%%%%%%%%%%%%%%%%%%%%%%%%%%%%%%%%%%%%%%%%%%%%%%%%%%%%%%%%%%%%%%%%%%%%
%%%%%%%%%%%%%%%%%%%%%%%%%%%%%%%%%%%%%%%%%%%%%%%%%%%%%%%%%%%%%%%%%%%%%%%%%%%%%%%%%%%%%%%%%%
\section{The case \texorpdfstring{$q=1$}{q=1}}
%\label{ss:q=1:setup}
\label{s:q=1}

The purpose of this subsection is to prove \cref{thm:main} \eqref{i:main:2}.
First, we give the exact statement in \cref{thm:q=1:main}.

%%%%%%%%%%%%%%%%%%%%%%%%%%%%%%%%%%%%%%%%%%%%%%%%%%%%%%%%%%%%%%%%%%%%%%%%%%%%%%%%%%%%%%%%%%
\subsection{The limit form under \texorpdfstring{$q \to 1$}{q to 1}}

As before, we consider the parameter set
\begin{align}
 \clN \ceq \{(n,m,k,l) \in \bbN^4 \mid m \le \fl{n/2}, \ k \le n, \ m+k-n \le l \le m \wedge k\}.
\end{align}
The function $p(x;q)=p(x \midd n,m,k,l;q)$ in \cref{dfn:0:p-genq} 
has a well-defined limit under $q \to 1$:
\begin{align}\label{eq:q=1:p}
\begin{split}
  p(x)
&=p(x \midd n,m,k,l) \ceq \lim_{q \to 1}p(x \midd n,m,k,l;q) \\
&=\binom{n-k}{m-l} \frac{\binom{n}{x}}{\binom{n}{m}} \frac{n-2x+1}{n-x+1} 
  \HG{4}{3}{-x,x-n-1,-M,-N}{-m,m-n,-M-N}{1}
\end{split}
\end{align}
with $M \ceq m-l$ and $N \ceq n-m-k+l$.
Here we used the standard symbols of binomial coefficients and 
hypergeometric series (see \cref{ss:0:ntn}).
We have $p(x) \in \bbQ$.

The terminating $\hg{4}{3}$-series in \eqref{eq:q=1:p} is Racah polynomial \cite[\S9.2]{KLS}.
In the definition \eqref{eq:qRacah} of $q$-Racah polynomial,
we replace $a=q^\alpha$, $b=q^\beta$, $c=q^\gamma$, $d=q^\delta$ and take the limit $q \to 1$.
Then we obtain Racah polynomial
\begin{align}
 R_s(\lambda(y);\alpha,\beta,\gamma,\delta) \ceq
 \HG{4}{3}{-s,s+\alpha+\beta+1,-y,y+\gamma+\delta+1}
          {\alpha+1,\beta+\delta+1,\gamma+1}{1}
\end{align}
with $\lambda(y) \ceq (y+\gamma+\delta+1)$.
The limit function $p(x)$ \eqref{eq:q=1:p} can be rewritten as 
\begin{align}\label{eq:q=1:p=R}
 p(x) = \binom{n-k}{m-l} \frac{\binom{n}{x}}{\binom{n}{m}} \frac{n-2x+1}{n-x+1} \, 
 R_x(\lambda(M); -m-1,-n+m-1,-n+k-1,0).
\end{align}
We call \eqref{eq:q=1:p=R} the Racah presentation of $p(x)$.

The formula for $p(x;q)$ in \cref{thm:sum1} has the following $q \to 1$ limit.

\begin{thm}[Hahn presentation] 
For $(n,m,k,l) \in \clN$ and $x \in \{0,1,\dotsc,m\}$, we have %the following identity in $\bbQ$.
\begin{align}\label{eq:q=1:4332}
 p(x) = \frac{\binom{n}{x}}{\binom{n}{m}} \, \frac{n-2x+1}{n-x+1} \, 
 \sum_{i=0}^{M \wedge N} \binom{M}{i} \binom{N}{i} \, \HG{3}{2}{-i,-x,x-n-1}{-m,m-n}{1}.
\end{align}
The terminating $\hg{3}{2}$-series in \eqref{eq:q=1:sph-omega} 
is Hahn polynomial \cite[\S9.5]{KLS}.
\end{thm}

The $q \to 1$ limit of \cref{thm:q:sum=1} yields the following summation formula of $p(x)$.

\begin{thm}\label{thm:q=1:sum=1}
For $(n,m,k,l) \in \clN$ and $x \in \{0,1,\dotsc,m\}$, we have %the following identity in $\bbQ$.
\begin{align}\label{eq:q=1:cdf}
 \sum_{u=0}^x p(u) = 
 \binom{n-k}{m-l} \frac{\binom{n}{x}}{\binom{n}{m}}
 \HG{4}{3}{-x,x-n,-M,-N}{-m,m-n,-M-N}{1}.
\end{align}
We also have $\sum_{u=0}^m p(u) = 1$.
\end{thm}

%We leave it for a fun exercise to check 
%$\sum_{u=0}^m p(u)=1$ directly using Chu-Vandermonde formula.

Now we recall the statement of \cref{thm:main} \eqref{i:main:2}.

\begin{thm}\label{thm:q=1:main}
For $(n,m,k,l) \in \clN$, 
there exists a discrete probability distribution $P_{n,m,k,l}$ on the set $\{0,1,\dotsc,m\}$
whose pmf is given by the function $p(x) = p(x \midd n,m,k,l)$ \eqref{eq:q=1:p}.
The cdf $P_{n,m,k,l}[X \le x]$ is equal to 
the terminating $\hg{4}{3}$-series \eqref{eq:q=1:cdf}.
\end{thm}

The proof is very similar to the case where $q$ is a prime power (\cref{s:q}),
but here we use the representation theory of the symmetric group $\frS_n$.
In \cref{ss:q=1:pos}, \eqref{eq:q=1:p(x)}, we introduce a new quantity $\wt{p}(x)$
using the $\frS_n$-action in the classical Schur-Weyl duality between $\frS_n$ and $\SU(2)$.
We have the positivity $\wt{p}(x) \ge 0$ \eqref{eq:q=1:p(x)}.
Next we give \cref{thm:Hahn}, which states that 
$\wt{p}(x)$ is equal to the right hand side of the equality \eqref{eq:q=1:4332} on $p(x)$.
Thus, we have $\wt{p}(x)=p(x)$, and hence $p(x) \ge 0$.
Since $\sum_{x=0}^mp(x)=1$ by \cref{thm:q=1:sum=1}, 
we have the first half of the claim.
The latter half follows from \cref{thm:q=1:sum=1}.

%%%%%%%%%%%%%%%%%%%%%%%%%%%%%%%%%%%%%%%%%%%%%%%%%%%%%%%%%%%%%%%%%%%%%%%%%%%%%%%%%%%%%%%%%%
\subsection{The positivity from \texorpdfstring{$\frS_n$}{Sn}-action} 
%distribution \texorpdfstring{$P_{n,m,k,l}$}{Pnmkl}}
\label{ss:q=1:pos}

In this subsection, we construct a positive quantity $\wt{p}(x)$ 
using some $\frS_n$-representation, and prove the positivity of $p(x)$.
The construction of $\wt{p}(x)$ follows a process similar to that of 
the $q$-case in \cref{ss:q:pos}, but with the group $\GL(n,\Fq)$ replaced by $\frS_n$, 
and the tensor space $T=(\bbC^2)^{\otimes [n]_q}$ by $(\bbC^2)^{\otimes n}$.

%Taking the limit $q \to 1$ in the $q$-setup (\cref{ss:q:pos}), 
%we are in the following situation.
Let us consider the tensor space $(\bbC^2)^{\otimes n}$
of $\bbC^2 = \bbC \ket{0} \oplus \bbC \ket{1}$.
The symmetric group $\frS_n$ acts on $(\bbC^2)^{\otimes n}$ by permuting tensor factors,
and this action commutes with that of the natural representation $\bbC^2$ of 
the special unitary group $\SU(2)$.
Thus, we are in the situation of the classical Schur-Weyl duality
$\SU(2) \ract (\bbC^2)^{\otimes n} \lact \frS_n$.
This bimodule decomposes as 
\begin{align}\label{eq:q=1:SW}
 (\bbC^2)^{\otimes n} = \bigoplus_{x=0}^{\fl{n/2}} U_{(n-x,x)} \otimes V_{(n-x, x)},
\end{align}
where the left factor $U_{(n-x,x)}$ is the highest weight $\SU(2)$-irrep of dimension $n-2x+1$, and
the right factor $V_{(n-x,x)}$ is the $\frS_n$-irrep corresponding to the partition $(n-x,x)$.

We equip with $\bbC^2$ the standard Hermitian form, 
which induces another Hermitian form on $(\bbC^2)^{\otimes n}$,
and this tensor space is a unitary representation of $\frS_n$.
Then, similarly as \eqref{eq:q:P_a}, we have the orthogonal projector
$\sfP_{(n-x,x)}\colon (\bbC^2)^{\otimes n} \srj U_{(n-x,x)} \otimes V_{(n-x,x)}$
expressed by
%Regarding it as the projector for $\frS_n$-representation, 
%we apply the standard formula to obtain
\begin{align}\label{eq:1-2}
 \sfP_{(n-x,x)} = \frac{\dim V_{(n-x,x)}}{n!}
 \sum_{g \in \frS_n} \ol{\chi_{(n-x,x)}(g)} \, \pi(g),
\end{align}
where $\chi_{(n-x,x)}$ denotes the character of $V_{(n-x,x)}$, 
the overline denotes the complex conjugate, and $\pi$ denotes the $\frS_n$-action.
The dimension of $V_{(n-x,x)}$ is given by the hook length formula
(see \cite[I.5, Example 2; I.7, (7.6)]{M} for example), i.e.,
the $q \to 1$ limit of \eqref{eq:q:dimV_x}:
\begin{align}\label{eq:q=1:dim}
 \dim V_{(n-x,x)} = \frac{n-2x+1}{n-x+1} \binom{n}{x} = \binom{n}{x}-\binom{n}{x-1}.
\end{align}
Now the same argument as in \eqref{eq:q:gen-pos} gives
\begin{align}\label{eq:q=1:gen-pos}
 \bra{\psi}\sfP_{(n-x,x)}\ket{\psi} \ge 0
\end{align}
for any $\ket{\psi} \in (\bbC^2)^{\otimes n}$ and $x\in\{0,1,\dotsc,\fl{n/2}\}$.

Next we make a distinguished choice of $\ket{\psi}$.
Motivated by the analysis of asymmetry in the context of quantum information theory 
(see \cite{HHY2,H3}), we consider the mixture of a coherent state and 
a randomized state in $(\bbC^2)^{\otimes n}$.
The coherence is described by the $\frS_n$-action, so that 
a coherent state is given by the normalized vector 
\begin{align}
\label{eq:q=1:1l0k}
 \ket{1^l \, 0^{k-l}} 
&\ceq 
 \ket{1}^{\otimes l} \otimes \ket{0}^{\otimes (k-l)} \in (\bbC^2)^{\otimes k}
\end{align}
with some $l,k-l \in \bbN$, and a randomized state is 
\begin{align} %\label{eq:0:Dicke}
 \ket{\Xi_{N+M,M}} 
&\ceq \tbinom{N+M}{M}^{-1/2} 
 \bigl(\ket{1^M \, 0^N} + \text{$\frS_{N+M}$-permuted terms}\bigr)
 \in (\bbC^2)^{\otimes (N+M)}
\end{align}
with some $M,N \in \bbN$.
In quantum physics literatures, the state $\ket{\Xi_{N+M,M}}$ is called the Dicke state \cite{Di}.
Now the mixture we want to consider is the normalized vector 
\begin{align}\label{eq:q=1:Xi}
 \ket{\Xi_{n,m|k,l}} \ceq \ket{1^l \, 0^{k-l}} \otimes \ket{\Xi_{N+M,M}} \in (\bbC^2)^{\otimes n},
\end{align}
where we set $M \ceq m-l$ and $N \ceq n-m-k+l$. 
Thus, $m$ denotes the number of $\ket{1}$'s and $N+M=n-k$ denotes the ``length'' of the Dicke state.
The range of the tuple $(n,m,k,l)$ is given by the conditions $l,k-l,M,N \in \bbN$, i.e., 
it belongs to the set $\clN'$ \eqref{eq:0:clN'}.

Now we assume $m \le n-m$ and take $(n,m,k,l) \in \clN$ \eqref{eq:0:clN}. 
For $x \in \{0,1,\dotsc,m\}$, we define 
%discrete probability distribution $P_{n,m,k,l}$ on $\{0,1,\dotsc,\fl{n/2}\}$
%arising from the element $\ket{\Xi_{n,m|k,l}}$. The pmf is given by 
\begin{align}\label{eq:q=1:p(x)}
 \wt{p}(x) = \wt{p}(x \midd n,m,k,l) \ceq 
 \bra{\Xi_{n,m|k,l}} \sfP_{(n-x,x)} \ket{\Xi_{n,m|k,l}}.
\end{align}
By \eqref{eq:q=1:gen-pos}, we have
\begin{align}\label{eq:q=1:pos}
 \wt{p}(x) \ge 0.
\end{align}

%\begin{rmk}
%We have the following symmetry of the quantity $\wt{p}(x)$.
%\begin{align}\label{eq:q=1:sym}
% \wt{p}(x \midd n,m,k,l) = \wt{p}(x \midd n,n-m,k,k-l).
%\end{align}
%In fact, by the $\SU(2)$-action, 
%we can flip the basis elements $\ket{0}$ and $\ket{1}$ of $\bbC^2$
%up to a scalar of absolute value $1$, and it maps $\ket{\Xi_{n,m|k,l}}$ to 
%$\ket{\Xi_{n,n-m|k,k-l}}$ up to a scalar of absolute value $1$.
%Since the actions of $\frS_n$ and $\SU(2)$ on $(\bbC^2)^{\otimes n}$ intertwine, 
%this map commutes with the projector $\sfP_{(n-x,x)}$.
%Then, by the definition \eqref{eq:q=1:p(x)} of $\wt{p}(x)$, 
%we have the equality \eqref{eq:q=1:sym}.
%
%Thus, it is enough to assume $m \le n-m$ and consider the restricted parameter set
%$\clN$ in \eqref{eq:0:clN}.
%The value of $x$ is then restricted to $\{0,1,\dotsc,m\}$.
%\end{rmk}

The element $\ket{\Xi_{n,m|k,l}}$ turns out to be a nice spherical vector,
and we can evaluate the value $\wt{p}(x)$. 
Hereafter we use the standard symbol $\hg{r}{s}$ of hypergeometric series \eqref{eq:ntn:hg}. 

\begin{thm}%[{Hahn presentation}]
\label{thm:Hahn}
For $(n,m,k,l) \in \clN$ and $x \in \{0,1,\dotsc,m\}$, we have
\begin{align}\label{eq:q=1:H}
 \wt{p}(x \midd n,m,k,l) = \frac{\binom{n}{x}}{\binom{n}{m}} \, \frac{n-2x+1}{n-x+1} \, 
 \sum_{i=0}^{M \wedge N} \binom{M}{i} \binom{N}{i} \, \omega_{(n-x,x)}(i), 
\end{align}
where $M:=m-l$, $N:=n-m-k+l$, and 
\begin{align} 
\label{eq:q=1:sph-omega}
  \omega_{(n-x, x)}(i) 
&=\HG{3}{2}{-i,-x,x-n-1}{-m,m-n}{1} 
 =\binom{m}{i}^{-1} \binom{n-m}{i}^{-1}
  \sum_{r=0}^{i \wedge x} (-1)^r \binom{x}{r} \binom{m-x}{i-r} \binom{n-m-x}{i-r}.
\end{align}
\end{thm}

\begin{proof}
We can compute $\wt{p}(x)$ in a similar way as the proof of \cref{thm:q:qH},
replacing the Gelfand pair $(\GL(n,\Fq),P(m,n-m,\Fq))$ by $(\frS_n,\frS_m \times \frS_{n-m})$
and every $q$-quantities in the formulas by their $q \to 1$ limit.
For example, $q$-binomial coefficients are replaced by binomial coefficients, and 
$\dim_{\bbC}V_x$ is replaced by $\dim_{\bbC}V_{(n-x,x)}$.
We refer to the preprint version \cite[\S4.1]{HHY} 
for a detailed explanation.
\end{proof}

Let us prove \cref{thm:q=1:main}. 
The above \cref{thm:Hahn} and the equality \eqref{eq:q=1:4332} imply $p(x)=\wt{p}(x)$.
Since $\wt{p}(x) \ge 0$ by \eqref{eq:q=1:p(x)}, we have $p(x) \ge 0$.
Since $\sum_{x=0}^mp(x)=1$ by \cref{thm:q=1:sum=1}, 
the function $p(x)$ is the pmf of a discrete probability distribution.
The cdf is given by the terminating $\hg{4}{3}$-series in \cref{thm:q=1:sum=1}.
Hence we obtain \cref{thm:q=1:main}, i.e., \cref{thm:main} \eqref{i:main:2}.

By \eqref{eq:q=1:p=R}, the pmf $p(x)$ is expressed by Racah polynomial,
and we can obtain various properties of $p(x)$, 
which will be used in the study of asymptotic behavior 
under the limit $n \to \infty$ in \cite{HHY2,H3}.
Below we give just one corollary, which is the key tool of our asymptotic analysis. 
It is an immediate consequence of the formula \eqref{eq:q=1:p=R}
and the three-term recursion of Racah polynomial \cite[\S 9.2, pp.191--192]{KLS}.

\begin{prp}[{Three-term recurrence relation}]
\label{prp:rec}
For $(n,m,k,l) \in \clN$ and $x \in \{0,1,\dotsc,m\}$, 
the pmf $p(x) = p(x \midd n,m,k,l)$ satisfies 
\begin{align}
\begin{split}
 a_x \frac{n-2x-1}{n-x} \binom{n}{x+1}^{-1} p(x+1) 
&-(a_x+c_x-M N) \frac{n-2x+1}{n-x+1} \binom{n}{x}^{-1} p(x) \\
&+c_x \frac{n-2x+3}{n-x+2} \binom{n}{x-1}^{-1} p(x-1) = 0,
\end{split}
\end{align}
where $M \ceq m-l$, $N \ceq n-m-k+l$, and the coefficients $a_x$ and $c_x$ are given by 
\begin{align}
 a_x &\ceq \frac{(m-x)(n-m-x)(n-k-x)(n-x+1)}{(n-2x)(n-2x+1)}, \\
 c_x &\ceq \frac{x(x-k-1)(m-x+1)(n-m-x+1)}{(n-2x+1)(n-2x+2)}.
\end{align}
\end{prp}

%%%%%%%%%%%%%%%%%%%%%%%%%%%%%%%%%%%%%%%%%%%%%%%%%%%%%%%%%%%%%%%%%%%%%%%%%%%%%%%%%%%%%%%%%%
%%%%%%%%%%%%%%%%%%%%%%%%%%%%%%%%%%%%%%%%%%%%%%%%%%%%%%%%%%%%%%%%%%%%%%%%%%%%%%%%%%%%%%%%%%
\section{The case \texorpdfstring{$0<\abs{q-1} \ll 1$}{0<|q-1|<<1}}
\label{s:q-small}

In this section, we study the case where $q$ is a real number 
with $q \ne 1$ and $\abs{q-1}$ sufficiently small,
and show \cref{thm:main} \eqref{i:main:3}.
Recall from \cref{dfn:0:p-genq} the function $p(x;q) \in \bbQ(q)$ 
depending on $(n,m,k,l) \in \clN$ and $x \in \{0,1,\dotsc,m\}$.
By \cref{thm:q:sum=1}, we have $\sum_{x=0}^m p(x;q)=1$.
Thus, to show the existence of distribution in \cref{thm:main} \eqref{i:main:3}, 
it remains to prove:

\begin{thm}\label{thm:pos}
Let $(n,m,k,l) \in \clN$ and $x \in \{0,1,\dotsc,m\}$.
For $q \in \bbR$ with $q \ne 0$ and $\abs{q-1}$ small enough, we have $p(x;q) \ge 0$.
\end{thm}

In the proof, we will use the property of the limit function $p(x)=\lim_{q\to1}p(x;q)$.
By \cref{thm:q=1:main}, we have $p(x) \ge 0$.
Since $p(x;q)$ is a continuous function of $q$, 
it is enough to show that for $x \in \{0,1,\dotsc,m\}$ satisfying $p(x)=0$,
we have $p(x;q)=0$.

Let us study the support of $p(x)$, i.e., the range of $x \in \{0,1,\dotsc,m\}$ such that $p(x)>0$.
The result is:

\begin{prp}\label{prp:p>0}
$p(x)>0$ if and only if $x \in \{0,1,\dotsc,k \wedge m\}$.
\end{prp}

\begin{proof}
For a while, we assume only $(n,m,k,l) \in \clN'$ \eqref{eq:0:clN'}, 
i.e., we do not assume $m \le n-m$.
We regard the bimodule $\SU(2) \ract (\bbC^2)^{\otimes n} \lact \frS_n$ as an $\SU(2)$-module 
(in contrast, we regard it as an $\frS_n$-module in the proof of \cref{thm:Hahn}).
We use the common symbols in physics literatures to denote by 
let $\{\ket{j,m}_{\SU(2)} \mid m=-j,-j+1,\dotsc,j\}$ the standard orthonormal basis of
the $(2j+1)$-dimensional $\SU(2)$-irrep with angular momenta 
$\js \ket{j,m}_{\SU(2)} = j(j+1) \ket{j,m}_{\SU(2)}$ and 
$\jz \ket{j,m}_{\SU(2)} = m \ket{j,m}_{\SU(2)}$.

First, we consider the second factor $\ket{\Xi_{n-k,m-l}}$ of 
$\ket{\Xi_{n,m|k,l}} = \ket{1^l 0^{k-l}} \otimes \ket{\Xi_{n-k,m-l}}$.
It is the $\frS_{n-k}$-averaged state of $\ket{1^{n-k}0^{m-l}}$, so we have 
\begin{align}\label{eq:CG:Xi}
 \ket{\Xi_{n-k,m-l}} = \ket{\tfrac{n-k}{2},\tfrac{n-k}{2}-(m-l)}_{\SU(2)} \in U_{(n-k,0)}, 
\end{align}
where, as in \eqref{eq:q=1:SW}, $U_{(n-k,0)}$ denotes the $\SU(2)$-irrep of dimension $n-k$.
Next, we consider the first factor $\ket{1^l 0^{k-l}} \in (\bbC^2)^{\otimes k}$.
The $\SU(2)$-decomposition of $(\bbC^2)^{\otimes k}$ is
\begin{align}% \label{eq:CG:Ck}
 (\bbC^2)^{\otimes k} 
 = \bigoplus_{u=0}^{\fl{k/2}} U_{(k-u,u)}^{\oplus \dim V_{(k-u,u)}}
 = \bigoplus_{u=0}^{\fl{k/2}} 
   U_{(k-u,u)}^{\oplus \left(\binom{k}{u}-\binom{k}{u-1}\right)},
\end{align}
where $U_{(k-u,u)}$ is the $\SU(2)$-irrep of dimension $k-2u+1$, 
$V_{(k-u,u)}$ is the $\frS_k$-irrep, 
and we used the hook length formula \eqref{eq:q=1:dim} of $\dim V_{(k-u,u)}$.
In the middle and right, we used the symbol $U^{\otimes n}$ to denote 
the $n$-th direct sum of $U$ as $\SU(2)$-representation.
Along this decomposition, we have 
%The element $\ket{1^l 0^{k-l}} \in (\bbC^2)^{\otimes k}$ is decomposed along \eqref{eq:CG:Ck} as
\begin{align} %\label{eq:CG:5}
 \ket{1^l 0^{k-l}} = \binom{k}{l}^{-1/2} \sum_{u=0}^{(k-l) \wedge l} \ 
 \sum_{i=1}^{\dim V_{(k-u,u)}} 
 \ket{\tfrac{k}{2}-u,\tfrac{k}{2}-l}_{\SU(2)}^{(i)},
\end{align}
where the $\SU(2)$-bases $\ket{\tfrac{k}{2}-u,\tfrac{k}{2}-l}_{\SU(2)}^{(i)}$ 
are orthogonal for different $i$, and the normalization factor comes from 
$\sum_{u=0}^{(k-l) \wedge l} \dim V_{(k-u,u)}=\binom{k}{l}$ by the hook length formula.

The tensor product $\ket{\Xi_{n,m|k,l}} = \ket{1^l 0^{k-l}} \otimes \ket{\Xi_{n-k,m-l}}$
is then calculated by the Littlewood-Richardson rule and the Clebsch-Gordan coefficient.
The former says 
%\begin{align} %\label{eq:CG:6}
$U_{(k-u,u)} \otimes U_{(n-k,0)} \cong 
 \bigoplus_{x=u}^{(k-u) \wedge (n-k+u)} U_{(n-x,x)}$, 
and the latter, denoted as $c(x,u \midd n,m,k,l) \ceq
  \pair{\tfrac{k}{2}-u,\tfrac{k}{2}-l,\tfrac{-k}{2},\tfrac{n-k}{2}-(m-l)}
       {\tfrac{n}{2}-x,\tfrac{n}{2}-m}$, gives 
\begin{gather*} %\label{eq:CG:7}
  \ket{\tfrac{k}{2}-u,\tfrac{k}{2}-l}_{\SU(2)} \otimes 
  \ket{\tfrac{n-k}{2}, \tfrac{n-k}{2}-(m-l)}_{\SU(2)} 
 =\sum_{x=u}^{(k-u) \wedge (n-k+u)} c(x,u \midd n,m,k,l) 
  \ket{\tfrac{n}{2}-x,\tfrac{n}{2}-m}_{\SU(2)}^{(i)}.
\end{gather*}
%\end{proof}

Collecting the above calculations and replacing the summation order, we have
%\eqref{eq:CG:Xi}, \eqref{eq:CG:5}, \eqref{eq:CG:6} and \eqref{eq:CG:7}, we have
\begin{align*}
  \ket{\Xi_{n,m|k,l}} 
&=\binom{k}{l}^{-1/2} \ \sum_{u=0}^{(k-l) \wedge l} \ 
  \sum_{x=u}^{(k-u) \wedge (n-k+u)} c(x,u \midd n,m,k,l)
  \sum_{i=1}^{\dim V_{(k-u,u)}} \ket{\tfrac{n}{2}-x,\tfrac{n}{2}-m}_{\SU(2)}^{(i)} \\
%\label{eq:CG:PXi}
&=\binom{k}{l}^{-1/2} \ \sum_{x=0}^{k \wedge m \wedge (n-m)} \ 
  \sum_{u = 0 \vee (x-n+k)}^{x \wedge (k-x) \wedge (k-l) \wedge l} c(x,u \midd n,m,k,l)
  \sum_{i=1}^{\dim V_{(k-u,u)}} \ket{\tfrac{n}{2}-x,\tfrac{n}{2}-m}_{\SU(2)}^{(i)}.
\end{align*}
Thus, the quantity $\sfP_{(n-x,x)} \ket{\Xi_{n,m|k,l}}$ is the $x$-term in the first summation, and
%c(x,u \midd n,m,k,l)
% \sum_{i=1}^{\dim V_{(k-u,u)}} \ket{\tfrac{n}{2}-x,\tfrac{n}{2}-m}_{\SU(2)}^{(i)}$ 
%c(x,u \midd n,m,k,l)
% \sum_{i=1}^{\dim V_{(k-u,u)}} \ket{\tfrac{n}{2}-x,\tfrac{n}{2}-m}_{\SU(2)}^{(i)}.
%\end{align}
%Since the basis $\ket{\tfrac{n}{2}-x,\tfrac{n}{2}-m}_{\SU(2)}$ exists 
%only for $\abs{\tfrac{n}{2}-m} \le \tfrac{n}{2}-x$, 
%the effective $x$ satisfies $x \le m \wedge (n-m)$.
%%By the system symmetry \eqref{eq:sym}, we also have $x \le m$.
%We also have $x \le k$ by \eqref{eq:CG:PXi}.
%Thus we obtain the vanishing condition \eqref{eq:van}.
%
%we have for $0 \le x \le k \wedge m \wedge (n-m)$ that
\begin{align}\label{eq:q=1:CG}
 p(x) = \bra{\Xi_{n,m|k,l}} \sfP_{(n-x,x)} \ket{\Xi_{n,m|k,l}} = 
 \sum_{u = 0 \vee (x-n+k)}^{ x \wedge (k-x) \wedge l \wedge (k-l)}
 \frac{\dim V_{(k-u,u)}}{\binom{k}{l}} c(x,u \midd n,m,k,l)^2. % \binom{k}{u}-\binom{k}{u-1}
\end{align}
Note that each term in the summation is non-negative.

Now we assume $(n,m,k,l) \in \clN$, so that $m \le n-m$.
Then we have $p(x)=0$ unless $0 \le x \le k \wedge m$.
In order to obtain the conclusion, we want to show $p(x)>0$ for $0 \le x \le k \wedge m$.
Thus, it is enough to show that there is a non-zero Clebsch-Gordan coefficient 
in the summation \eqref{eq:q=1:CG}.
But, it contains the term (i) $u=x$ or (ii) $u=k-x$ or (iii) $u=l \wedge (k-l)$.
If we denote the general Clebsch-Gordan coefficient by $\pair{j_1m_1j_2m_2}{j_3 m_3}$,
then (i) and (ii) correspond to the case $j_3=\abs{j_1 \pm j_2}$, 
and (iii) to the case $m_1=\pm j_1$.
Thus, these terms are on the ``boundary'' of the selection rule.
According to \cite[\S8.1.1, (6)]{VMK}, the corresponding Clebsch-Gordan coefficients are non-zero.
Hence we have the consequence.

%The symbol $\pair{\tfrac{k}{2}-u, \dotsc}{\tfrac{n}{2}-x, \tfrac{n}{2}-m}$ 
%denotes the Clebsch-Gordan coefficient. 
%In general, the Clebsch-Gordan coefficient 
%$\pair{j_1m_1j_2m_2}{j_3(-m_3)}$ is converted to 3j symbol as 
%\begin{align}\label{eq:CG:3j}
% \begin{pmatrix} j_1 & j_2 & j_3 \\ m_1 & m_2 & m_3\end{pmatrix} = 
% \frac{(-1)^{j_1-j_2-m_3}}{\sqrt{2j_3+1}} \pair{j_1m_1j_2m_2}{j_3(-m_3)}.
%\end{align}
%Then, the coefficient $\pair{\tfrac{k}{2}-u, \dotsc}{\tfrac{n}{2}-x, \tfrac{n}{2}-m}$ is zero 
%unless $\frac{n}{2}-u \ge \frac{n}{2}-x \ge \frac{1}{2} \abs{n-2k+2u}$ and $n-2x \ge \abs{n-2m}$.
%Hence we have the range $u \le x \le (k-u) \wedge (n-k+u) \wedge m \wedge (n-m)$
%in \eqref{eq:CG:a11}. 
\end{proof}

\begin{rmk}
The expression \eqref{eq:q=1:CG} is obviously non-negative, 
in comparison to the Hahn and Racah presentations (\cref{thm:Hahn}, \eqref{eq:q=1:p=R}) 
which are alternating summations.
However, the Clebsch-Gordan symbol is a $\hg{3}{2}$-polynomial, 
so that \eqref{eq:q=1:CG} is a triple sum, and not suitable for further analysis.
\end{rmk}

We return to the proof of \cref{thm:pos}.
%Due to the discussion right after the statement, 
It remains to show:

\begin{lem}\label{lem:q;p=0}
For $(n,m,k,l) \in \clN$ and $x \in \{0,1,\dotsc,m\}$, the following equality holds in $\bbQ(q)$.
\begin{align}
 p(x;q) = 0 \quad (x>k).
\end{align}
\end{lem}

\begin{proof}
Recall Sears' $\qhg{4}{3}$ transformation formula \cite[(3.2.1)]{GR}:
\begin{align}\label{eq:Sears}
 \qHG{4}{3}{q^{-s},a,b,c}{d,e,f}{q}{q} = 
 \frac{(e/a,de/bc;q)_s}{(e,de/abc;q)_s} \qHG{4}{3}{q^{-s},a,d/b,d/c}{d,de/bc,q^{1-s}a/e}{q}{q}
\end{align}
for $s \in \bbN$ and $abc=defq^{s-1}$, i.e., 
both sides are terminating balanced $\qhg{4}{3}$ series. 
Using $M \ceq m-l$ and $N \ceq n-m-k+l$, the function $p(x;q)$ corresponds to the case
\[
 (s,a,b,c,d,e,f)=(x,q^{x-n-1},q^{-M},q^{-N},q^{-m},q^{m-n},q^{-M-N}), 
\]
and in the right hand side of \eqref{eq:Sears}, 
we have the term $(de/bc;q)_s=(q^{-k};q)_x$, which vanishes for $x>k$. 
One can check that the terms in the denominator do not vanish.
Hence, we have the consequence.
\end{proof}

\cref{thm:pos} now follows from \cref{lem:q;p=0} and \cref{prp:p>0}.

Finally, let us write down the proof of \cref{thm:main} \eqref{i:main:3}.
Since we have $\sum_{x=0}^m p(x;q)=1$ by \cref{thm:q:sum=1} 
and $p(x;q) \ge 0$ by \cref{thm:pos},
the function $p(x;q)$ is the pmf of a discrete probability distribution on $\{0,1,\dotsc,m\}$.
The cdf is given by the $\qhg{4}{3}$-series in \cref{thm:q:sum=1}.
Hence we have \cref{thm:main} \eqref{i:main:3}.

The proof of the main \cref{thm:main} is now finished.

%%%%%%%%%%%%%%%%%%%%%%%%%%%%%%%%%%%%%%%%%%%%%%%%%%%%%%%%%%%%%%%%%%%%%%%%%%%%%%%%%%%%%%%%%%
%%%%%%%%%%%%%%%%%%%%%%%%%%%%%%%%%%%%%%%%%%%%%%%%%%%%%%%%%%%%%%%%%%%%%%%%%%%%%%%%%%%%%%%%%%
\section{Discussion}\label{s:dsc}

In the main \cref{thm:main}, we showed that the function $p(x;q)$ in \cref{dfn:0:p-genq} 
and the limit function $p(x) \ceq \lim_{q\to1}p(x;q)$ are the pmf in the cases where 
(1) $q$ is a prime power, (2) $q=1$, and (3) $q \in \bbR$ with $0 < \abs{q-1} \ll 1$.
In particular, we have $p(x;q) \ge 0$ in these cases.
It remains as a big problem to completely determine
the range of $q \in \bbR$ where $p(x;q) \ge 0$ holds.

Our distribution $P_{n,m,k,l;q}$ is encoded by $q$-Racah polynomial \cref{eq:0:p=qR},
and the orthogonality of $q$-Racah polynomial gives the three-term recursion 
(c.f.\ \cref{prp:rec} for $q=1$ case) which will be used extensively in our sequel \cite{HHY2} 
of the asymptotic analysis.
Since $q$-Racah polynomial is in the top row of the $q$-Askey scheme 
of basic hypergeometric polynomials, it seems difficult to add an extra parameter to 
$P_{n,m,k,l;q}$ while retaining these good properties.

In recent years, there has been an increasing number of studies on 
elliptic and theta hypergeometric series \cite[Chapter 11]{GR}.
As for Racah polynomials, there are known elliptic 6j symbols,
which are further generalization of quantum 6j symbols \cite{KiR}, 
first appeared in Boltzmann weights of fused eight-vertex SOS model \cite{D+}.
As a hypergeometric series, elliptic 6j symbols can be written 
by ${}_{12} V_{11}$ series \cite[\S11.2]{GR}.
As for orthogonality, Spiridonov and Zhedanov \cite{SZ,SZ2} studied elliptic 6j symbols
in the point of view of biorthogonal functions.
See \cite{Ro} for the details and references. 
%and the approach using Sklyanin algebra (an elliptic variation of quantum group).

In contrast to these non-polynomial elliptic analogs,
a polynomial analogue is found in the recent paper \cite{vDG},
where the starting point of the discussion is the difference Heun equation.
This polynomial satisfies the three-term recurrence and the orthogonality.

Now it is tempting to think of an elliptic analogue of our distribution $P_{n,m,k,l;q}$.
The polynomial in \cite{vDG} may be a clue to this problem,
although there do not seem to be any hypergeometric-type summation formulas
for the polynomial, which was the key ingredient in this article. 
Another possible clue is the Gelfand pair $(\GL(n,\Fq),P(m,n-m,\Fq))$ or the Grassmannian 
$\Gr(m,n)$ over $\bbF_q$ which we used to construct $P_{n,m,k,l;q}$.
Is there an elliptic analogue of such a ``classical $q$-object''?
If it exists, then, is there an elliptic analogue of the theory of spherical functions?

%%%%%%%%%%%%%%%%%%%%%%%%%%%%%%%%%%%%%%%%%%%%%%%%%%%%%%%%%%%%%%%%%%%%%%%%%%%%%%%%%%%%%%%%%%
%%%%%%%%%%%%%%%%%%%%%%%%%%%%%%%%%%%%%%%%%%%%%%%%%%%%%%%%%%%%%%%%%%%%%%%%%%%%%%%%%%%%%%%%%%
\begin{Ack}
M.\ H.\ was supported in part by the National Natural Science Foundation of China under Grant 62171212.
A.\ H.\ was supported in part by JSPS Grant-in-Aids for Scientific Research Grant Number 19K03532. 
S.\ Y.\ was supported in part by JSPS Grant-in-Aids for Scientific Research Grant Number 19K03399.
A part of this article is presented in the 16th OPFSA
(International Symposium on Orthogonal Polynomials, Special Functions and Application).
The authors thank the organizers for the opportunity.
They would also like to thank the referees for valuable comments and suggestions on improvements.
In particular, as explained in \cref{rmk:ref1,rmk:ref2,rmk:ref3}, 
the statement of \cref{thm:main} \eqref{i:main:1} 
and some of the arguments in \cref{ss:q:pos} are added after their suggestions.
Finally, S.Y.\ thanks Professor Masatoshi Noumi for valuable comments on spherical functions
and hypergeometric functions.
\end{Ack}

%%%%%%%%%%%%%%%%%%%%%%%%%%%%%%%%%%%%%%%%%%%%%%%%%%%%%%%%%%%%%%%%%%%%%%%%%%%%%%%%%%%%%%%%%%
%%%%%%%%%%%%%%%%%%%%%%%%%%%%%%%%%%%%%%%%%%%%%%%%%%%%%%%%%%%%%%%%%%%%%%%%%%%%%%%%%%%%%%%%%%


\begin{thebibliography}{MMMN}

\bibitem[AAR99]{AAR}
G.\ E.\ Andrews, R.\ Askey, R.\ Roy, 
\emph{Special functions}, 
\emph{Encyclopedia of Mathematics and its Applications}, \textbf{71}, 
Cambridge Univ.\ Press, Cambridge, 1999.

%Racah
\bibitem[AW79]{AW1}
R.\ Askey, R.\ Wilson, 
\emph{A Set of Orthogonal Polynomials That Generalize the Racah Coefficients %
 or 6-j Symbols}, SIAM J.\ Math.\ Anal.\ \textbf{10} (1979), no.\ 5, 1008--1016.

%Racah
\bibitem[AW85]{AW2}
R.\ Askey, R.\ Wilson, 
\emph{Some basic hypergeometric orthogonal polynomials %
 that generalize Jacobi polynomials},
Mem.\ Amer.\ Math.\ Soc.\ \textbf{319}, Amer.\ Math.\ Soc.\ Providence, RI, 1985.

\bibitem[BI84]{BI}
E.\ Bannai, T.\ Ito, 
\emph{Algebraic combinatorics I.\ Association schemes}, 
Benjamin/Cummings Publishing Co., Inc., Menlo Park, CA, 1984.

\bibitem[B+85]{LNM}
C.\ Brezinski, A.\ Draux, A.\ P.\ Magnus, P.\ Maroni, A.\ Ronveaux, ed.,
\emph{Polyn\^{o}mes orthogonaux et applications}, 
%(Proceedings of the Laguerre symposium held at Bar-le-Duc, October 15--18, 1984),
Lecture Notes in Math., \textbf{1171}. Springer-Verlag, Berlin, 1985.

%\bibitem[CP94]{ChP}
%V.\ Chari, A.\ Pressley, 
%\emph{A Guide to Quantum Groups}, Cambridge University Press, Cambridge, 1994.

\bibitem[CP16]{CP}
I.\ Corwin, L.\ Petrov,
\emph{Stochastic Higher Spin Vertex Models on the Line},
Comm.\ Mat.\ Phys., \textbf{343} (2016), 651--700.

\bibitem[CR62]{CR}
C.\ W.\ Curtis, I.~Reiner, 
\emph{Representation theory of finite groups and associative algebras},
Pure and Applied Mathematics, Vol. XI, Interscience Publishers, 
a division of John Wiley \& Sons, New York-London, 1962;
reprint from AMS Chelsea Publishing, Providence, RI, 2006.

\bibitem[D+86]{D+}
E.\ Date, M.\ Jimbo, T.\ Miwa, M.\ Okado,
\emph{Fusion of the eight vertex SOS model}, 
Lett.\ Math.\ Phys., \textbf{12} (1986), 209--215; 
Erratum and addendum, Lett.\ Math.\ Phys., \textbf{14} (1987), 97.

%\bibitem[D+88]{D+2}
%E.\ Date, M.\ Jimbo, A.\ Kuniba, T.\ Miwa, M.\ Okado.
%\emph{Exactly solvable SOS models II. Proof of the star-triangle relation and combinatorial identities}. 
%in: M.\ Jimbo, et al.\ (eds.) \emph{Conformal Field Theory and Solvable Lattice Models}, pp.\ 17--122,
%Adv.\ Stud.\ Pure Math.. \textbf{16}, Academic Press, 1988.


\bibitem[De73]{De1}
P.~Delsarte,
\emph{An algebraic approach to the association schemes of coding theory},
Philips Res.\ Rep.\ (1973), no.\ 10. %Suppl.

\bibitem[De78]{De2}
P.~Delsarte,
\emph{Hahn Polynomials, Discrete Harmonics, and $t$-Designs},
SIAM J.\ Appl.\ Math., \textbf{34} (1978), no.\ 1, 157--166.

%Hayashi10-17
%\bibitem{Delsarte}
%P.\ Delsarte, P.\ Piret, 
%\emph{Algebraic constructions of Shannon codes for regular channels},
%IEEE Trans.\ Inform.\ Theory \textbf{28} (1982), no.\ 4, 593--599. 
%Jul. 1982.

%Hayashi Motivation 1
\bibitem[Di54]{Di}
R.\ H.\ Dicke,
\emph{Coherence in Spontaneous Radiation Processes},
Phys.\ Rev.\ \textbf{93} (1954), 99--110.

\bibitem[Dr86]{Dr}
V.\ G.\ Drinfel'd, \emph{Quantum groups}, in \emph{Proceedings of the International Congress of
Mathematicians, Berkeley, 1986}, A.\ M.\ Gleason (ed.), 798--820, 
Amer.\ Math.\ Soc., Providence, RI.

%\bibitem[Du78a]{Du1}
%C.~F.~Dunkl,
%\emph{An addition theorem for Hahn polynomials: the spherical functions},
%SIAM J.\ Math.\ Anal.\ \textbf{9} (1978), no.\ 4, 627--637.

%App A
\bibitem[Du78b]{Du2}
C.~F.~Dunkl,
\emph{An addition theorem for some q-Hahn polynomials},
Monatsh.\ Math.\ \textbf{85} (1978), no.\ 1, 5--37.

%\bibitem[Du79]{Du3}
%C.~F.~Dunkl, \emph{Orthogonal functions on some permutation groups},
%in \emph{Relations between combinatorics and other parts of mathematics}, 129--147,
%Proc.\ Sympos.\ Pure Math., XXXIV, Amer.\ Math.\ Soc., Providence, R.I., 1979.

%\bibitem[EV99]{EV}
%P.\ Etingof, A.\ Varchenko, 
%\emph{Exchange dynamical quantum groups}, 
%Comm.\ Math.\ Phys., \textbf{205} (1999), 19--52.

\bibitem[F95]{F}
G.\ Felder, 
\emph{Elliptic quantum groups}, 
Proc.\ ICMP, Paris-1994 (1995), 211--218.

%\bibitem[FV95]{FV}
%G.\ Felder, A.\ Varchenko, 
%\emph{On representations of the elliptic quantum group $E_{\tau,\eta}(sl2)$},
%Comm.\ Math.\ Phys., \textbf{181} (1996), 741--761.

\bibitem[GR04]{GR}
G.~Gasper, M.~Rahman, 
\emph{Basic hypergeometric series}, 2nd ed.,
Encyclopedia of Mathematics and its Applications, \textbf{96},
Cambridge Univ.\ Press, Cambridge, 2004.

\bibitem[HHY]{HHY}
M.\ Hayashi, A.\ Hora, S.\ Yanagida,
\emph{Asymmetry of tensor product of asymmetric and invariant vectors arising from 
 Schur-Weyl duality based on hypergeometric orthogonal polynomial}, 
preprint (2021), 71pp.; arXiv:2104.12635v1.

\bibitem[HHY2]{HHY2}
M.\ Hayashi, A.\ Hora, S.\ Yanagida,
\emph{Stochastic behavior of outcome of Schur-Weyl duality measurement}, 
preprint (2023), 45pp.; arXiv:2104.12635v2.

\bibitem[H3]{H3}
M.\ Hayashi, 
\emph{Coherence in permutation-invariant state enhances permutation-asymmetry}, 
preprint (2023), 18pp; arXiv:2311.10307,

%\bibitem{Ho}
%A.~Hora,
%\emph{Representations of Symmetric Groups and Asymptotic Analysis on Young diagrams}
%(in Japanese), S\={u}gaku no Yashiro \textbf{4}, S\={u}gaku Syob\={o}, 2017.

%\bibitem[Ho98]{Ho}
%A.\ Hora,
%\emph{Central limit theorems and asymptotic spectral analysis on large graphs},
%Infin.\ Dimens.\ Anal.\ Quantum Probab.\ Relat.\ Top.\ 
%\textbf{1} (1998), no.\ 2, 221--246.
%
%\bibitem[HO07]{HO}
%A.\ Hora, N.\ Obata, 
%\emph{Quantum probability and spectral analysis of graphs}, 
%Theoretical and Mathematical Physics, Springer, 2007.

%App.A
\bibitem[J86]{J}
M.\ Jimbo,
\emph{A $q$-analogue of $U(\mathfrak{gl}(N+1))$, Hecke algebra, and the Yang-Baxter equation},
Lett.\ Math.\ Phys., \textbf{11} (1986), no.\ 3, 247--252.

\bibitem[YB89]{YB}
M.\ Jimbo (ed.),
\emph{Yang-Baxter equation in integrable systems},
Adv.\ Ser.\ Math.\ Phys., \textbf{10}, World Scientific, 1989.

\bibitem[KR89]{KiR}
A.\ N.\  Kirillov, N.\ Yu.\ Reshetikhin, 
\emph{Representations of the algebra $U_q(sl(2))$, 
 $q$-orthogonal polynomials and invariants of links}, 
in: V.\ G.\ Kac, (ed.) \emph{Infinite-dimensional Lie Algebras and Groups}, pp.\ 285--339.
World Sci. Publ. Co., Teaneck, NJ, 1989.

\bibitem[KLS10]{KLS}
R.\ Koekoek, P.\ A.\ Lesky, R.\ F.\ Swarttouw,
\emph{Hypergeometric Orthogonal Polynomials and Their $q$-Analogues},
Springer Monographs in Mathematics, Springer-Verlag Berlin Heidelberg, 2010.

\bibitem[KR01]{KR}
E.\ Koelink, H.\ Rosengren,
\emph{Harmonic Analysis on the $\SU(2)$ Dynamical Quantum Group},
Acta Appl.\ Math., \textbf{69} (2001), 163--220.

%Hayashi10-2
%\bibitem{Korzekwa}
%K.~Korzekwa, Z.~Pucha{\l}a, M.~Tomamichel, K.~{\.Z}yczkowski,
%\emph{Encoding classical information into quantum resources}, preprint, 2019; 
%\url{https://arxiv.org/abs/1911.12373}

%\bibitem{KH}
%W.~Kumagai, M.~Hayashi, 
%\emph{Second-Order Asymptotics of Conversions of Distributions and %
% Entangled States Based on Rayleigh-Normal Probability Distributions}, 
%IEEE Trans.\ Inform.\ Theory \textbf{63} (2017), Issue 3, 1829--1857. 

%Racah1
%\bibitem[Le82]{Le}
%D.\ A.\ Leonard, 
%\emph{Orthogonal polynomials, duality and association schemes}, 
%SIAM J.\ Math.\ Anal.\ \textbf{13} (1982), 656--663.

%\bibitem[L93]{L}
%G.\ Lusztig, 
%\emph{Introduction to quantum groups},
%Progr.\ Math., \textbf{110}, Birkh\"auser, 1993.

\bibitem[M95]{M} 
I.\ G.\ Macdonald,
\emph{Symmetric Functions and Hall Polynomials}, 2nd ed.,
Oxford Univ.\ Press, New York, 1995.

%CMP referee
\bibitem[M14]{Man}
V.\ V.\ Mangazeev, 
\emph{On the Yang-Baxter equation for the six-vertex model},
Nucl.\ Phys.\ \textbf{B 882} (2014), 70--96.


%App.A
\bibitem[MP07]{MP}
J.\ M.\ Marco, J.\ Parcet, 
\emph{Laplacian Operators and Radon Transforms on Grassmann Graphs},
Monatsh.\ Math.\ \textbf{150} (2007), 97--132.

%Hayashi10-1
%\bibitem{Marvian}
%I.~Marvian, \emph{Symmetry, asymmetry and quantum information},
%Ph.D.\ dissertation, University of Waterloo, 2012; %[Online]. 
%Available from UWSpace \url{https://uwspace.uwaterloo.ca/handle/10012/7088}

%Hayashi10-6
%\bibitem{Petz}
%D.~Petz.
%\emph{An invitation to the algebra of canonical commutation relations},
%Leuven notes in Mathematical and Theoretical Physics,
%Series A: Mathematical Physics \textbf{2}, 
%Leuven Univ.\ Press, Leuven, 1990.

\bibitem[Ra42]{Ra}
G.\ Racah, \emph{Theory of complex spectra II}, Phys.\ Rev., \textbf{62}, 438--462, 1942.

%\bibitem{R}
%C.~R.~Rao,
%\emph{Linear statistical inference and its applications}, 
%John Wiley \& Sons, Inc., New York-London-Sydney, 1965.

\bibitem[Ro07]{Ro}
H.\ Rosengren,
\emph{An elementary approach to $6j$-symbols 
 (classical, quantum, rational, trigonometric, and elliptic)},
Ramanujan J., \textbf{13} (2007), 131--166.

\bibitem[S66]{S}
L.\ J.\ Slater, \emph{Generalized Hypergeometric Functions}, 
Cambridge Univ.\ Press, Cambridge, 1966.

%\bibitem{Sp}
%J.~Spencer, \emph{Asymptopia}, with L.~Florescu, 
%Student Mathematical Library, \textbf{71}. 
%Amer.\ Math.\ Soc., Providence, RI, 2014.

%\bibitem[R02]{SR}
%K.\ S.\ Rao, %Srinivasa
%\emph{Quantum Theory of Angular Momentum: Hypergeometric Series and Polynomial zeros},
%Surikaiseki K\^oky\^uroku, \textbf{1274} (2002), 12--28.

\bibitem[SZ00]{SZ}
V.\ P.\ Spiridonov, A.\ S.\ Zhedanov,
\emph{Spectral transformation chains and some new biorthogonal rational functions},
Comm.\ Math.\ Phys., \textbf{210} (2000), 49--83.

\bibitem[SZ01]{SZ2}
V.\ P.\ Spiridonov, A.\ S.\ Zhedanov,
\emph{Generalized eigenvalue problems and 
 a new family of rational functions biorthogonal on elliptic grids}. in: 
Bustoz, J., et al.\ (eds.), 
\emph{Special Functions 2000: Current Perspective and Future Directions}, pp. 365--388. 
Kluwer Acad. Publ., Dordrecht (2001)

\bibitem[vDG22]{vDG}
J.\ F.\ van Diejen, T.\ G\"orbe,
\emph{Elliptic Racah polynomials},
Lett.\ Math.\ Phys., \textbf{111} (2022), 66, 27 pp.

\bibitem[VMK88]{VMK}
D.~A.~Varshalovich, A.~N.~Moskalev, V.~K.~Khersonski\u{i}, 
\emph{Quantum theory of angular momentum}, %Translated from the Russian. 
World Scientific Publishing Co., Inc., 
Teaneck, NJ, 1988.

\bibitem[Wig65]{Wig}
E.\ P.\ Wigner, 
\emph{On the matrices which reduce the Kronecker products of representations of S.\ R.\ groups},
manuscript (1940); published in: L.\ C.\ Biedenharn, H.\ Van Dam (eds.), 
\emph{Quantum Theory of Angular Momentum}, pp. 87--133. Academic Press, New York, 1965.

\bibitem[W78]{W}
J.\ A.\ Wilson,
\emph{Hypergeometric series, recurrence relations and some new orthogonal functions}, 
PhD Thesis, University of Wisconsin, Madison, 1978.
\end{thebibliography}
\end{document}